\documentclass[12pt]{amsart}
\usepackage[margin=1.0in]{geometry}
\usepackage{amsmath,amssymb,amsfonts,graphicx,color,fancyhdr,psfrag,comment,enumerate,mathabx}
\usepackage[latin1]{inputenc}
\usepackage[hyperpageref]{backref}
\usepackage[colorlinks=true, pdfstartview=FitV, linkcolor=blue, citecolor=blue, urlcolor=blue]{hyperref}
\usepackage[capitalise,noabbrev]{cleveref}
\usepackage{tcolorbox,longfbox}
\allowdisplaybreaks
\usepackage{epstopdf}
\usepackage{mathrsfs}
\usepackage{epsfig}
\usepackage{hyperref}

\usepackage{ifthen}
\usepackage{float}
\usepackage{enumerate}
\usepackage{mathtools}
\usepackage[bottom]{footmisc}
\usepackage{tikz}
\usepackage{comment}
\usetikzlibrary{matrix,shapes,arrows,positioning,chains}
\usepackage{amssymb,amsmath,amsfonts}
\usepackage{bbm}

\numberwithin{equation}{section}

\newtheorem{theorem}{Theorem}[section]
\newtheorem{proposition}[theorem]{Proposition}
\newtheorem{lemma}[theorem]{Lemma}
\newtheorem{definition}[theorem]{Definition}

\newtheorem{remark}[theorem]{Remark}

\newcommand{\R}{\mathbb{R}}

\newcommand{\N}{\mathbb{N}}

\begin{document}

\title[Besov and Triebel--Lizorkin spaces]{On Besov and Triebel--Lizorkin spaces associated with the Grushin operator}

\author[S. Bagchi, N. Garg, and R. Garg]
{Sayan Bagchi, Nishta Garg \and Rahul Garg}

\address[S. Bagchi]{Department of Mathematics and Statistics, Indian Institute of Science Education and Research Kolkata, Mohanpur--741246, West Bengal, India.}
\email{sayan.bagchi@iiserkol.ac.in}

\address[N. Garg]{Department of Mathematics, Indian Institute of Science Education and Research Bhopal, Bhopal--462066, Madhya Pradesh, India.}
\email{nishta21@iiserb.ac.in}

\address[R. Garg]{Department of Mathematics, Indian Institute of Science Education and Research Bhopal, Bhopal--462066, Madhya Pradesh, India.}
\email{rahulgarg@iiserb.ac.in}

\subjclass[2020]{Primary: 46E36. Secondary: 22E25, 42B20, 58J35}
\keywords{Grushin operator with drift, Exponentially growing measure, Besov spaces, Triebel--Lizorkin spaces, Embeddings, Fractional Leibniz rule.}

\begin{abstract}
In this article, we define classical and non-classical Besov and Triebel--Lizorkin spaces associated with the Grushin operator with or without drift. We establish various characterisations of these spaces, and study their complex interpolation, embedding properties, and fractional Leibniz rules. 
\end{abstract}

\maketitle

\section{Introduction}

In the theory of function spaces, there is a vast literature on Besov and Triebel--Lizorkin spaces. These spaces are known to be quite useful in the study of partial differential equations. For classical results on the Euclidean space, one can refer to \cite{Theory_of_function_spaces}. Owing to their wide applicability, there has also been extensive research in recent times on such function spaces beyond the Euclidean space. In particular, we refer to \cite{characterisations_of_Besov_and_Triebel_Lizorkin_spaces, A_theory_of_Besov_and_Triebel_Lizorkin_spaces_modeled_on_CC_distance, A_difference_characterisation_of_Besov_and_Triebel_Lizorkin_spaces_on_RD_spaces, Continuous_characterisations_of_inhomogeneous_Besov_and_Triebel--Lizorkin_spaces_associated_to_non-negative_self-adjoint_operators, Feneuil_Joseph_Algebra_properties_2018, Furioli_LPD_and_Besov_spaces_on_Lie_groups_of_polynomial_growth_2006, Gallagher_Besov_algebras_on_Lie_groups_of_polynomial_growth_2012} which consist of works related to doubling metric measure spaces, Lie groups with a left-invariant Riemannian structure, and Lie groups endowed with a sub-Riemannian structure. 

\medskip 
In the present article, we study non-homogeneous classical and non-classical Besov and Triebel--Lizorkin spaces associated with the Grushin operator with or without drift. We remark that stronger results for homogeneous function spaces are recently established in \cite{Bruno-Homogeneous-algebras-via-heat-kernel-estimates-2022, bui2025bilinearfractionalleibnizrules,Bui_Huy_Duong_weightes_Besov_and_Triebel_Lizorkin_spaces_2020} and the setup in these works is much more general. However, these works only deal with classical function spaces and that too without a non-trivial drift. Building upon the recent work of Bruno et al. \cite{Bruno_Marco_Vallarino_Besov_TL_20}, where the authors introduced non-homogeneous Besov and Triebel--Lizorkin spaces associated with sub-Laplacians with drift on certain non-compact, connected Lie groups, in this article we study the above-mentioned function spaces associated with the Grushin operator with or without drift in a unified way. 

\medskip 
In the Euclidean setup, Besov and Triebel--Lizorkin spaces are defined by using the Littlewood--Paley decomposition of a function, which depends on the Mihlin--H\"{o}rmander's multiplier theorem for the Laplacian. In the non-Euclidean setup, the Mihlin--H\"{o}rmander's multiplier theorem is known to fail in some cases. This, in fact, is also the case for the sub-Laplacian with drift. To circumvent this issue, the authors of \cite{Bruno_Marco_Vallarino_Besov_TL_20} utilised the classical ``Gauss Weierstrass'' characterisation of such spaces to define analogous spaces in their setup. Let us briefly recall it. 

\medskip 
Let $G$ be a non-compact connected Lie group with the set of left-invariant vector fields $\{X_{1}, \ldots, X_{l}\}$ satisfying H\"{o}rmander's condition. Let $\sigma$ be the right Haar measure on $G$ and $\delta$ the modular function. Denote by $\chi$ a positive and continuous character on $G$ and let $d\mu_{\chi} = \chi \, d\sigma$. Let 
$$
\Delta_{\chi} = -\sum_{j=1}^{l}(X_{j}^{2} + c_{j} X_{j}), \qquad  \text{where} \, c_{j} = (X_{j}\chi)(e), \, \, \text{for} \, \, j=1, \ldots, l.
$$
It is known that $\Delta_{\chi}$ is essentially self-adjoint on $L^{2}(\mu_{\chi})$ (see \cite[Proposition 3.1]{Hebisch-Mauceri-Meda-spectral-multipliers-drift-Lie-group-Math-Z-2005}). Now, for any $\alpha \geq 0, \, m > \alpha/2, \, p, \, q \in [1,\infty]$ and $t_0 \in (0,1)$, authors in \cite{Bruno_Marco_Vallarino_Besov_TL_20} defined Besov and Triebel--Lizorkin spaces to be the spaces of all tempered distributions $f$ such that
$$
\|e^{-t_0\Delta_{\chi}}f\|_{L^{p}(d\mu_{\chi})} + \left(\int_{0}^{1} (t^{-\alpha/2} \, \|(t\Delta_{\chi})^{m} \, e^{-t\Delta_{\chi}} f\|_{L^{p}(d\mu_{\chi})})^{q} \, \frac{dt}{t}\right)^{1/q} < \infty,
$$
and 
$$
\|e^{-t_0\Delta_{\chi}}f\|_{L^{p}(d\mu_{\chi})} + \left\|\left(\int_{0}^{1} (t^{-\alpha/2} \, |(t\Delta_{\chi})^{m} \, e^{-t\Delta_{\chi}} f|)^{q} \, \frac{dt}{t}\right)^{1/q}\right\|_{L^{p}(d\mu_{\chi})} < \infty,
$$
respectively, with obvious modification when $q = \infty$. They proved several equivalent characterisations of Besov and Triebel--Lizorkin norms, such as the Littlewood--Paley type characterisation, characterisation in terms of vector fields, and recursive characterisation. Furthermore, they studied their complex interpolation and embedding and algebra properties. 

\medskip 
Now, recall that the Grushin operator $G = -\Delta_{x'} - |x'|^2 \, \Delta_{x''}$, is a hypoelliptic operator on $\R^{d_1 + d_2}$. 
Here, $x = (x', \, x'') \in \R^{d_1} \times \R^{d_2}$. Recently, in \cite{Garg-Garg-Riesz-transform-Grushin-drift-2026}, we defined the Grushin operator with drift $G_\nu$ for non-zero vectors $\nu \in \R^{d_1}$ and studied boundedness properties of Riesz transforms associated with these operators on $L^p(d\mu_\nu)$-spaces. Afterwards, in \cite{GG-preprint-Sobolev-Grushin-2026}, we defined Sobolev spaces associated with $G$ and $G_\nu$ and studied their embedding and algebra properties. Continuing, in the present article we define Besov and Triebel--Lizorkin spaces and study properties of these spaces analogous to those studied in \cite{Bruno_Marco_Vallarino_Besov_TL_20}. 

\medskip
We defer the basic details pertaining to the operator $G_\nu$ to Section \ref{sec:preliminaries}. Also, when $\nu = 0$, $G_\nu$ is nothing but the Grushin operator $G$. Let $\mathcal{S}(\R^{d_1+d_2})$ stand for the space of all smooth functions $f$ on $\R^{d_1+d_2}$ such that 
\begin{align} \label{def:test_functions}
\sup_{x \in \R^{d_1+d_2}} e^{k|x'|} \, |X^{\gamma} f(x)| < \infty,
\end{align}
for all $k \in \N_0$ and all multi-indices $\gamma$. Here $X$ stands for the first-order gradient vector fields associated with the Grushin operator (see Subsection \ref{subsec:Grushin_operator} for the definition). Let us also denote by $\mathcal{S}'(\R^{d_1+d_2})$ the dual of $\mathcal{S}(\R^{d_1+d_2})$.

\medskip
In order to define our function spaces, we need to introduce the following operator. For any $m \in \N_{0}$ and $t>0$, let us consider 
\begin{align} 
\label{def-discrete-der}
W_{t}^{(m)} f = (tG_{\nu})^{m} e^{-tG_{\nu}} f. 
\end{align} 
We denote the ball volume of the open ball $B(x, \sqrt{t})$ by $V_\nu(x,t)$ (see 
\eqref{main:ball-vol-est}). In the following definitions and the rest of the article, we consider the weight functions $V_\nu(t)$ which are given by $V_\nu(t)(x) = V_\nu(x,t)$. In particular, when $t=1$, we shall write $V_\nu$ for $V_\nu(1)$. 

\begin{definition}[Classical Besov space] \label{def:Besov_space_classical}
Let $1 \leq p, \, q \leq \infty, \, \alpha \geq 0$ and $\eta \in \R$. The classical Besov space $B^{p,q}_{\alpha, \eta}(d\mu_{\nu})$ is defined to be the space of distributions $f \in \mathcal{S}'(\R^{d_1+d_2})$ for which 
$$
\|f\|_{B^{p,q}_{\alpha, \eta}(d\mu_{\nu})} := \mathcal{B}^{p,q}_{\alpha, \eta} (f) + \|V_{\nu}^{\eta} \, e^{-\frac{1}{2} G_{\nu}} f\|_{L^{p}(d\mu_{\nu})} < \infty,
$$
where 
\begin{align*}
\mathcal{B}^{p,q}_{\alpha, \eta} (f) &:= \left(\int_{0}^{1} \|t^{-\alpha/2} \, V_{\nu}^{\eta} \, |W_{t}^{([\alpha/2]+1)} f|\|^{q}_{L^{p}(d\mu_{\nu})} \, \frac{dt}{t} \right)^{1/q}, \quad \text{for } 1 \leq q < \infty, \\ 
\text{and} \qquad 
\mathcal{B}^{p,\infty}_{\alpha, \eta} (f) &:=  \sup_{t \in (0,1)} \|t^{-\alpha/2} \, V_{\nu}^{\eta} \, |W_{t}^{([\alpha/2]+1)} f|\|_{L^{p}(d\mu_{\nu})}. 
\end{align*}
\end{definition}

\medskip 
\begin{definition}[Classical Triebel--Lizorkin space] \label{def:Triebel_Lizorkin_space_classical}
Let $1 \leq p, \, q \leq \infty, \, \alpha \geq 0$ and $\eta \in \R$. The classical Triebel--Lizorkin space $F^{p,q}_{\alpha, \eta}(d\mu_{\nu})$ is defined to be the space of distributions $f \in \mathcal{S}'(\R^{d_1+d_2})$ for which 
$$
\|f\|_{F^{p,q}_{\alpha, \eta}(d\mu_{\nu})} := \mathcal{F}^{p,q}_{\alpha, \eta} (f) + \|V_{\nu}^{\eta} \, e^{-\frac{1}{2} G_{\nu}}  f\|_{L^{p}(d\mu_{\nu})} < \infty, 
$$ 
where 
\begin{align*}
\mathcal{F}^{p,q}_{\alpha, \eta} (f) &:= \left\|\left(\int_{0}^{1} \left( t^{-\alpha/2} \, V_{\nu}^{\eta} \,  |W_{t}^{([\alpha/2]+1)} f| \right)^{q} \, \frac{dt}{t} \right)^{1/q} \right\|_{L^{p}(d\mu_{\nu})}, \quad \text{for } 1 \leq q < \infty, \\ 
\text{and} \qquad 
\mathcal{F}^{p,\infty}_{\alpha, \eta} (f) &:= \left\|\sup_{t \in (0,1)} \left( t^{-\alpha/2} \,V_{\nu}^{\eta} \, |W_{t}^{([\alpha/2]+1)} f| \right)\right\|_{L^{p}(d\mu_{\nu})}.
\end{align*}
\end{definition}


\medskip 
\begin{definition}[Non-classical Besov space] 
\label{def:Besov_space}
Let $1 \leq p, \, q \leq \infty, \, \alpha \geq 0$ and $\eta \in \mathbb{R}$. The non-classical Besov space $\tilde{B}^{p,q}_{\alpha,\eta}(d\mu_{\nu})$ is defined to be the space of distributions $f \in \mathcal{S}'(\R^{d_1+d_2})$ for which 
$$
\|f\|_{\tilde{B}^{p,q}_{\alpha,\eta}(d\mu_{\nu})} := \mathcal{\Tilde{B}}^{p,q}_{\alpha, \eta} (f) + \|V_{\nu}^{-\alpha/Q+\eta} e^{-\frac{1}{2} G_{\nu}} f\|_{L^{p}(d\mu_{\nu})} < \infty, 
$$
where \begin{align*}
\mathcal{\tilde{B}}^{p,q}_{\alpha, \eta} (f) &:= \left(\int_{0}^{1} \|V_{\nu}(t)^{-\alpha/Q} \, V_{\nu}^{\eta} \, |W_{t}^{([\alpha/2]+1)} f|\|^{q}_{L^{p}(d\mu_{\nu})} \, \frac{dt}{t} \right)^{1/q}, \quad \text{for } 1 \leq q < \infty,  \\
\text{and} \qquad 
\mathcal{\Tilde{B}}^{p,\infty}_{\alpha, \eta} (f) &:=  \sup_{t \in (0,1)} \|V_{\nu}(t)^{-\alpha/Q} \, V_{\nu}^{\eta} \, |W_{t}^{([\alpha/2]+1)} f|\|_{L^{p}(d\mu_{\nu})}.
\end{align*}
\end{definition}

\medskip 
\begin{definition}[Non-classical Triebel--Lizorkin space]  \label{def:Triebel_Lizorkin_space}
Let $1 \leq p, \, q \leq \infty, \, \alpha \geq 0$ and $\eta \in \mathbb{R}$. The non-classical Triebel--Lizorkin space $\tilde{F}^{p,q}_{\alpha,\eta}(d\mu_{\nu})$ is defined to be the space of distributions $f \in \mathcal{S}'(\R^{d_1+d_2})$ for which 
$$
\|f\|_{\tilde{F}^{p,q}_{\alpha,\eta}(d\mu_{\nu})} := \mathcal{\tilde{F}}^{p,q}_{\alpha, \eta} (f) + \|V_{\nu}^{-\alpha/Q+\eta} e^{-\frac{1}{2} G_{\nu}}  f\|_{L^{p}(d\mu_{\nu})} < \infty,
$$
where 
\begin{align*}
\mathcal{\tilde{F}}^{p,q}_{\alpha, \eta} (f) &:= \left\|\left(\int_{0}^{1} \left( V_{\nu}(t)^{-\alpha/Q} \, V_{\nu}^{\eta} \, |W_{t}^{([\alpha/2]+1)} f| \right)^{q} \, \frac{dt}{t} \right)^{1/q} \right\|_{L^{p}(d\mu_{\nu})}, \quad \text{for } 1 \leq q < \infty,  \\
\text{and} \qquad 
\mathcal{\tilde{F}}^{p,\infty}_{\alpha, \eta} (f) &:= \left\|\sup_{t \in (0,1)} \left( V_{\nu}(t)^{-\alpha/Q} \, V_{\nu}^{\eta} \, |W_{t}^{([\alpha/2]+1)} f| \right) \right\|_{L^{p}(d\mu_{\nu})}.
\end{align*}
\end{definition}

One may observe that it follows from the local doubling property that  
$$ V_{\nu}(t)^{-\alpha/Q} \lesssim t^{-\alpha/2} \, V_{\nu}^{-\alpha/Q}, $$
for all $0 < t < 1$. Thus, each classical function space $B^{p, q}_{\alpha, \eta}(d\mu_{\nu})$ (resp. $F^{p, q}_{\alpha, \eta}(d\mu_{\nu})$) trivially embeds in the respective non-classical function space $\tilde{B}^{p, q}_{\alpha, \eta + \alpha/Q}(d\mu_{\nu})$ (resp. $\tilde{F}^{p, q}_{\alpha, \eta + \alpha/Q}(d\mu_{\nu})$). But we don't think there is any reverse side embedding. 

\medskip 
We have considered a factor of $V_{\nu}^{\eta}$ in the definitions of both classical and non-classical function spaces, which can also be seen as a weight function whenever $p < \infty$. Let us explain the motivation and importance of introducing such a factor. Recall that when working with function spaces on Lie groups as in \cite{Sobolev_embedding_Lie_groups, Bruno_Marco_Vallarino_Besov_TL_20}, thanks to the fact that the Haar measure of the ball $B(x,1)$ is independent of the center $x$, the ball volume $\mu_{\chi}(B(x, 1))$ is nothing but equivalent to $\chi(x)$ (where $\chi$ is the character associated with the drift). In particular, when $\chi$ is non-trivial, if we do not introduce an appropriate power of the character, one may not get embedding results of the function spaces with the exponentially growing $d \mu_{\chi}$ measure. This was already observed in the study of embeddings of Sobolev spaces in  \cite{Sobolev_embedding_Lie_groups} (see discussion leading to Theorem 1.1 in \cite{Sobolev_embedding_Lie_groups}) and again in \cite{Bruno_Marco_Vallarino_Besov_TL_20} (see discussion leading to Theorem 5.1 in \cite{Bruno_Marco_Vallarino_Besov_TL_20}). In Theorem 1.3 of our work \cite{GG-preprint-Sobolev-Grushin-2026} on Sobolev spaces associated with $G_\nu$, we took care of this phenomenon by incorporating the appropriate ball volume factor. This motivated us to bring in the weight function $V_{\nu}^{\eta}$ in the definitions of the function spaces in the present work.

\medskip 
In the next section, we shall begin with recalling most of the preliminary details relevant to our work. After that, in Section \ref{sec:spaces-definition-and-characterisation}, we shall study several norm characterisations of our Besov and Triebel--Lizorkin spaces, including independence of certain parameters, a Littlewood--Paley type characterisation, characterisation in terms of vector fields, and recursive characterisation. 
Let us also mention that in a recent work, Zhao et al. \cite{Geometric_topics_related_to_Besov_type_spaces_on_the_Grushin_setting_2005} also defined and studied classical Besov spaces for the Grushin operator. We shall compare our classical Besov space with that of \cite{Geometric_topics_related_to_Besov_type_spaces_on_the_Grushin_setting_2005} in Subsection \ref{subsec:comparison_of_classical_Besov_spaces}. Next, we discuss the complex interpolation properties of our function spaces in Section \ref{sec:interpolation}. Finally, in the last two Sections \ref{sec:embeddings} and \ref{sec:algebra_properties}, we study embedding properties and the fractional Leibniz rules for these spaces.


\section{Preliminaries} 
\label{sec:preliminaries}

\subsection{Grushin operator}
\label{subsec:Grushin_operator}

Let us write $x = (x', x'') \in \R^{d_1+d_2}$, where $x' = (x'_1, \ldots, x'_{d_1}) \in \R^{d_1}$ and $x'' = (x''_1, \ldots, x''_{d_2}) \in \R^{d_2}$, and consider the following vector fields on $\R^{d_1+d_2}$: 
$$
X_j = \frac{\partial}{\partial x'_j} \qquad \text{and} \qquad X_{j,k}= x'_j \, \frac{\partial}{\partial x''_k}, \qquad 1 \leq j \leq d_1, \, \hspace{1mm}  1\leq k\leq d_2.
$$
The Grushin operator is defined as the negative of the sum of the squares of these vector fields. More precisely, 
$$
G = -\sum_{j=1}^{d_1} X_j^2 - \sum_{j=1}^{d_1} \sum_{k=1}^{d_2} X_{j,k}^2 = -\sum_{j=1}^{d_1} \frac{\partial^2}{\partial x^{'2}_j} - |x'|^2 \sum_{k=1}^{d_2} \frac{\partial^2}{\partial x^{''2}_k}.
$$
It is well known that $G$ is a second-order hypoelliptic operator. Observe also that it is homogeneous of degree two with respect to the family of non-isotropic dilations $(\delta_r)_{r > 0}$ defined as $\delta_{r}(x', x'') = (rx', r^2 x'')$. The control distance $\rho$ of $G$ is known to satisfy the following asymptotics (see \cite{Analysis_of_degenerate_elliptic_opertators_Robinson}): 
\begin{equation} \label{Grushin_distance}
\rho(x,y) \sim |x'-y'| + \left\{
\begin{array}{ll}
\frac{|x''-y''|}{|x'|+|y'|}, & \mbox{if } |x''-y''|^{\frac{1}{2}} < |x'|+|y'| \\
|x''-y''|^{\frac{1}{2}},& \mbox{if } |x''-y''|^{\frac{1}{2}} \geq |x'|+|y'|.
\end{array} \right.    
\end{equation}
For convenience, we shall always refer to the right-hand side as the Grushin metric $\rho$.

\medskip Let $B(x,r) = \, \{y \in \R^{d_1+d_2}: \rho(x,y)<r\}$ denote the open ball with center $x$ and radius $r$. It was shown in \cite{Analysis_of_degenerate_elliptic_opertators_Robinson} that 
\begin{equation} \label{eq:ball_volume}
|B(x,r)| \sim  r^{d_1+d_2} \, \text{max}\{r,|x'|\}^{d_2} \sim  r^{d_1+d_2}(r+|x'|)^{d_2}, 
\end{equation}
where $|\cdot|$ denotes the Lebesgue measure on $\R^{d_1+d_2}$. Clearly, $(\R^{d_1+d_2}, \rho)$ with the Lebesgue measure $dx$ is a doubling metric measure space with the doubling constant $Q = d_1+2d_2$. It is easy to verify that there is a $C > 0$ such that
\begin{align} \label{est:center_change_in_ball_volume}
|B(x,r)| \leq C \left(1+\frac{\rho(x,y)}{r}\right)^{Q} |B(y,r)|,    
\end{align}
for all $x, \, y \in \R^{d_1+d_2}$ and $r > 0$.

\medskip The operator $G$ generates a symmetric diffusion heat semigroup $(e^{-tG})_{t > 0}$ on $L^{2}(dx)$, which is given by the following integral form: 
\begin{align} 
\label{integral-form-heat-semigp-Grushin}
e^{-tG} f(x) = \int_{\R^{d_1+d_2}} H_{t}(x,y) \, f(y) \, dy.
\end{align}
The heat kernel $H_{t}$ and its derivatives satisfy the following  Gaussian bounds (see \cite{Analysis_of_degenerate_elliptic_opertators_Robinson}). Given any $N \in \N$, there exist constants $b', \, b'' > 0$ such that 
\begin{align} 
|B(x,\sqrt{t})|^{-1} \, e^{-b'' \rho(x,y)^2/t} \lesssim H_{t}(x,y) & \lesssim |B(x,\sqrt{t})|^{-1} \, e^{-b' \rho(x,y)^2/t}, \label{est:heat_kernel_bounds} \\
|X^{\gamma}H_{t}(x,y)| & \lesssim  t^{-|\gamma|/2} \, |B(x, \sqrt{t})|^{-1} \, e^{-b' \rho(x,y)^2/t}, \label{est:heat_kernel_gradient_bounds} 
\end{align}
for all $x, \, y \in \R^{d_1+d_2}$, $t > 0$ and all multi-indices $\gamma \in \N_{0}^{d_1+d_1d_2}$ such that $|\gamma| \leq N$. Here, $X = (X', X'')$, with $X' = (X_{1}, \ldots, X_{d_1})$ and $X'' = (X_{1,1}, \ldots, X_{d_1,d_2})$. 

\medskip 
\textbf{Notation:} With some large $N$ fixed in mind, which may vary at different places, we shall always use the notation $b = b''/b'$ in this article. Clearly, $b \geq 1$. 


\subsection{Grushin operator with drift} \label{subsec:Grushin_operator_with_drift}
Let $\nu$ be a non-zero vector in $\R^{d_1}$. Introduced in \cite{Garg-Garg-Riesz-transform-Grushin-drift-2026}, the Grushin operator with drift is defined as follows: 
\begin{equation} \label{def:Grushin_operator_with_drift}
G_{\nu} = G - 2 \, \nu \cdot X' = -\sum_{j=1}^{d_1} \frac{\partial^2}{\partial x^{'2}_j} - |x'|^2 \sum_{k=1}^{d_2} \frac{\partial^2}{\partial x^{''2}_k} - 2 \sum_{j=1}^{d_1} \nu_{j} \frac{\partial}{\partial x'_j}.     
\end{equation}
It is positive-definite and essentially self-adjoint on $L^{2}(\R^{d_1+d_2}, d\mu_{\nu})$, where $d\mu_{\nu} = e^{2 \nu \cdot x'} \, dx$.

\medskip 
The measure $d\mu_{\nu}$ is in fact of exponential volume growth (see \cite[Lemma 2.4]{Garg-Garg-Riesz-transform-Grushin-drift-2026}). More precisely, if we denote the ball volume of the open ball $B(x, \sqrt{r})$ in this measure by $ V_\nu(x,r)$, then $V_\nu$ satisfies the following asymptotics: 
\begin{align}
\label{main:ball-vol-est}
V_\nu(x,r) \sim 
\left\{
\begin{array}{ll}
e^{2\nu \cdot x'} \, r^{(d_1+d_2)/2} \, (\sqrt{r}+|x'|)^{d_2}, & \mbox{if } \sqrt{r} \leq 1/|\nu| \\
|\nu|^{-(d_1+1)/2-d_2} \, e^{2(\nu \cdot x'+|\nu|\sqrt{r})} \, r^\frac{d_1-1}{4} \, (\sqrt{r}+|x'|)^{d_2}, & \mbox{if } \sqrt{r} > 1/|\nu|.
\end{array} \right.  
\end{align} 

The Grushin operator with drift $G_{\nu}$ generates a symmetric diffusion heat semigroup $(e^{-tG_{\nu}})_{t > 0}$, which has an integral form given by 
\begin{align*}
e^{-tG_{\nu}}f(x) = \int_{\R^{d_1+d_2}} H_{t, \, \nu} (x,y) \, f(y) \, d\mu_{\nu}(y),   
\end{align*}
with the heat kernel $H_{t, \nu}$ given explicitly by 
\begin{equation} \label{eq:heat_kernel_with_drift}
H_{t, \, \nu}(x, y) = e^{-t|\nu|^2} \, e^{-\nu\cdot(x'+y')} \, H_{t}(x,y).
\end{equation}
Here $H_{t} = H_{t,0}$ is the heat kernel of the Grushin operator $G$ as in \eqref{integral-form-heat-semigp-Grushin}. 

\medskip 
Henceforth, unless otherwise stated, we shall use the common notation $G_\nu$ to denote both the Grushin operator with drift for non-zero vectors $\nu$ as well as the Grushin operator $G_0 = G$. Note also that $V_0(x,r)$ is nothing but the Lebesgue measure of $B(x, \sqrt{r})$. 


\subsection{Some basic results}
\label{subsec:basic_results}
\medskip 
In this subsection, we collect or prove some basic results that will be used in various technical arguments later on. We begin by proving the following pointwise estimates concerning the heat semigroup $e^{-tG_{\nu}}$.  

\begin{lemma} \label{lem:heat_operator_and_derivative}
The heat semigroup $e^{-t G_{\nu}}$ satisfies the following estimates.
\begin{enumerate}[(i)]
\item Fix $0 < \kappa' < \kappa$. We have 
\begin{equation*} 
|e^{-t G_\nu} f| \, \lesssim \, e^{-\kappa b t_0 G_\nu} |f|, 
\end{equation*}
for every $t_0 \in (0,1)$ and $t \in [\kappa' t_0, \kappa t_0]$. 

\medskip 
\item Let $C > 0$. Given any $k \in \N_{0}$ and $\gamma \in \N_{0}^{d_1+d_1d_2}$, we have
\begin{align*} 
|X^{\gamma} e^{-t G_{\nu}} f| \, & \lesssim \, t^{-|\gamma|/2} e^{-b t G_\nu} |f|, \\
\text{and} \qquad |G_{\nu}^{k} e^{-t G_{\nu}} X^{\gamma} f| \, & \lesssim t^{-(k + |\gamma|/2)} e^{-b t G_\nu} |f|,
\end{align*} 
for all $t \in (0,C)$. 
\end{enumerate}
\end{lemma}
\begin{proof}
For part $(i)$, note that in view of relation \eqref{eq:heat_kernel_with_drift} and the heat kernel estimates \eqref{est:heat_kernel_bounds}, we have 
\begin{align*}
|e^{-t G_\nu} f(x)| & \leq \int_{\R^{d_1+d_2}} H_{t, \nu}(x,y) \, |f(y)| \, d\mu_{\nu}(y) \\
& \lesssim \int_{\R^{d_1+d_2}} e^{-t|\nu|^2} \, e^{-\nu \cdot (x'+y')} \, |B(x, \sqrt{t})|^{-1} \, e^{-b' \rho(x,y)^2/t} \, |f(y)| \, d\mu_{\nu}(y) \\
& \lesssim \int_{\R^{d_1+d_2}} e^{- \kappa b t_0|\nu|^2} \, e^{-\nu \cdot (x'+y')} \, |B(x, \sqrt{\kappa b t_0})|^{-1} \, e^{-b'' \rho(x,y)^2/(\kappa b t_0)} \, |f(y)| \, d\mu_{\nu}(y) \\
& \lesssim e^{-\kappa b t_0 G_\nu} |f|(x),
\end{align*}
for every $t \in [\kappa' t_0, \kappa t_0]$. 

\medskip 
Next, for the first estimate in part $(ii)$, restricting $t \in (0,C)$, we again make use of the relation \eqref{eq:heat_kernel_with_drift} and the estimates \eqref{est:heat_kernel_bounds} and \eqref{est:heat_kernel_gradient_bounds}, to get 
\begin{align*}
|X^{\gamma} e^{-t G_\nu} f(x)| & \leq \int_{\R^{d_1+d_2}} |X^{\gamma} H_{t, \nu}(x,y)| \, |f(y)| \, d\mu_{\nu}(y) \\
& = \int_{\R^{d_1+d_2}} \left|\sum_{\substack{\beta \leq \gamma\\\gamma_j = \beta_j \, \forall j > d_1}} \binom{\gamma}{\beta} \, (-\nu)^{\gamma-\beta} \, e^{-t|\nu|^2} \, e^{-\nu \cdot (x'+y')} X^{\beta} H_{t}(x,y) \right| \, |f(y)| \, d\mu_{\nu}(y) \\
& \lesssim \sum_{\substack{\beta \leq \gamma\\\gamma_j = \beta_j \, \forall j > d_1}} \int_{\R^{d_1+d_2}} e^{-t|\nu|^2} \, e^{-\nu \cdot (x'+y')} \, t^{-|\beta|/2} \, |B(x, \sqrt{t})|^{-1} \, e^{-b' \rho(x,y)^2/t} \, |f(y)| \, d\mu_{\nu}(y) \\
& \lesssim_{C} \int_{\R^{d_1+d_2}} e^{-b t|\nu|^2} \, e^{-\nu \cdot (x'+y')} \, t^{-|\gamma|/2} \, |B(x, \sqrt{bt})|^{-1} \, e^{-b'' \rho(x,y)^2/(bt)} \, |f(y)| \, d\mu_{\nu}(y) \\
& \lesssim t^{-|\gamma|/2} \int_{\R^{d_1+d_2}} H_{bt, \nu}(x,y) \, |f(y)| \, d\mu_{\nu}(y) \\
& = t^{-|\gamma|/2} e^{-bt G_\nu} |f|(x).
\end{align*}

Finally, one can prove the claimed estimate for $|G_{\nu}^{k} e^{-t G_{\nu}} X^{\gamma} f|$ using integration by parts, heat kernel bounds, and the same arguments as above.
\end{proof}

As we have seen in the introduction, our definitions of Besov and Triebel--Lizorkin spaces involve ball volumes. It is obvious that the pointwise multiplication by ball volume does not commute with the heat semigroup, and the following lemma provides a tool to handle this situation. We shall make use of this trick every now and then in many of our arguments. 

\medskip 
Given a measurable function $f$, let us write $\Tilde{f}_{t,s}(x) = V_{\nu}(x, t)^{s} f(x)$, for $t>0$ and $s \in \mathbb{R}$. 
\begin{lemma} \label{lemma:ball_volume_and_heat_operator_commutation}
Let $s \in \R$ and $C > 0$. We have
\begin{align*}
V_{\nu}(x, t)^{s} \left| e^{-t G_\nu}f(x) \right| & \lesssim e^{-2btG_\nu} |\Tilde{f}_{t,s}| (x), \\
V_{\nu}(x, 1)^{s} \left| e^{-t G_\nu}f(x) \right| & \lesssim e^{-2btG_\nu} |\Tilde{f}_{1,s}| (x), \\
|e^{-t G_\nu} \Tilde{f}_{t,s} (x)| & \lesssim V_{\nu}(x, t)^{s} \, e^{-2btG_\nu}|f|(x), \\ 
\text{and} \qquad \qquad |e^{-t G_\nu} \Tilde{f}_{1,s} (x)| & \lesssim V_{\nu}(x, 1)^{s} \, e^{-2btG_\nu}|f|(x), 
\end{align*}
for all $0 < t \leq C$. 
\end{lemma}
\begin{proof}
Recall from asymptotics \eqref{eq:ball_volume} and \eqref{main:ball-vol-est} that 
\begin{align*}
V_{\nu}(x, t) \sim 
\left\{
\begin{array}{ll}
e^{2\nu \cdot x'} \, |B(x,\sqrt{t})|, & \mbox{if } \sqrt{t} \leq 1/|\nu| \\ 
(|\nu| \sqrt{t})^{-\frac{(d_1+1)}{2} - d_2} \, e^{2(\nu \cdot x' + |\nu| \sqrt{t})} \, |B(x,\sqrt{t})|, & \mbox{if } \sqrt{t} > 1/|\nu|.
\end{array} \right.  
\end{align*} 
We shall only prove the first estimate. All other estimates can be proved in a similar manner. Also, we shall only consider the case when $\sqrt{t} \leq 1/|\nu|$, as the other case can be handled in a similar way. Now, making use of identity \eqref{eq:heat_kernel_with_drift} and estimates \eqref{est:center_change_in_ball_volume} and \eqref{est:heat_kernel_bounds}, we have 
\begin{align*}
& V_{\nu}(x, t)^{s} \, |e^{-t G_\nu} \, f(x)| \\ 
& \leq \int V_{\nu}(x, t)^{s} \, H_{t,\nu}(x,y) \, |f(y)| \, d\mu_{\nu}(y) \\
& \lesssim \int e^{2 s \nu \cdot x'} \, |B(x,\sqrt{t})|^{s} \, e^{-t|\nu|^2} \, e^{- \nu \cdot (x'+y')} \, |B(x, \sqrt{t}) |^{-1} \, e^{-b' \rho(x,y)^2/t} \, |f(y)| \, d\mu_{\nu}(y) \\
& \lesssim \int e^{2s\nu \cdot x'} \left(1+ \frac{\rho(x,y)}{\sqrt{t}}\right)^{|s|Q} |B(y,\sqrt{t})|^{s} \, e^{-t|\nu|^2} \, e^{- \nu \cdot (x'+y')} \, |B(x, \sqrt{t} ) |^{-1} \, e^{-b' \rho(x,y)^2/t} \, |f(y)| \, d\mu_{\nu}(y) \\
& \sim \int e^{2s \nu \cdot (x'-y')} \left(1+ \frac{\rho(x,y)}{\sqrt{t}}\right)^{|s|Q} \, e^{-t|\nu|^2} \, e^{- \nu \cdot (x'+y')}  |B(x, \sqrt{t})|^{-1} \,  e^{-b'\rho(x,y)^2/t} \, |\Tilde{f}_{t,s}(y)| \, d\mu_{\nu}(y) \\ 
& \lesssim_C \int e^{2s \nu \cdot (x'-y')} \left(1+ \frac{\rho(x,y)}{\sqrt{t}}\right)^{|s|Q} \, e^{-2bt|\nu|^2} \, e^{- \nu \cdot (x'+y')}  |B(x, \sqrt{2bt})|^{-1} \,  e^{-b'\rho(x,y)^2/t} \, |\Tilde{f}_{t,s}(y)| \, d\mu_{\nu}(y) \\ 
& \lesssim \int e^{\frac{2|s| |\nu| \rho(x,y) \sqrt{C}}{\sqrt{t}}} \left(1+ \frac{\rho(x,y)}{\sqrt{t}}\right)^{|s|Q} \, e^{-b'\rho(x,y)^{2}/(2t)} \, H_{2bt, \nu}(x,y) \, |\Tilde{f}_{t,s}(y)| \, d\mu_{\nu}(y), 
\end{align*}
for every $0 < t \leq C$, and since 
\begin{align*}
e^{\frac{2|s| |\nu| \rho(x,y) \sqrt{C}}{\sqrt{t}}} \left(1+ \frac{\rho(x,y)}{\sqrt{t}}\right)^{|s|Q} \, e^{-b'\rho(x,y)^{2}/(2t)} 
\lesssim_{C, s} 1,
\end{align*}
we get
\begin{align*} 
|V_{\nu}(x, t)^{s} \, e^{-t G_\nu} \, f(x)| & \lesssim_{C, s} \int H_{2bt, \nu}(x,y) \,  |\Tilde{f}_{t,s}(y)| \, d\mu_{\nu}(y) = e^{-2bt G_\nu} |\Tilde{f}_{t,s}|(x).  
\end{align*}
This completes the proof of the claimed inequality. 
\end{proof}

The following result is about $L^p-L^q$-boundedness of the heat semigroup. 
\begin{proposition} \label{prop:L^p_L^q_type_ineq_for_heat_operator}
Let $1 \leq p \leq q \leq \infty$. Given any $C>0$, we have 
$$
\|V_{\nu}(\cdot, t)^{\frac{1}{p}-\frac{1}{q}} \, e^{-tG_{\nu}} f\|_{L^{q}(d\mu_{\nu})} \lesssim_{C} \|f\|_{L^{p}(d\mu_{\nu})},
$$
for all $0 < t \leq C$. 
\end{proposition}
\begin{proof}
Note that when $p = q$, the result follows from the known $L^{p}(d\mu_{\nu})$-boundedness of the heat operator $e^{-tG_{\nu}}$. 

\medskip Now, let us consider the case when $p = 1$ and $1 \leq q < \infty$. In this case, 
\begin{align}
& \nonumber \|V_{\nu}(\cdot, t)^{1-\frac{1}{q}} \, e^{-tG_{\nu}}f\|_{L^{q}(d\mu_{\nu})} \\
& \nonumber \qquad = \left(\int_{\R^{d_1+d_2}} \left|V_{\nu}(x, t)^{1-\frac{1}{q}} \, \int_{\R^{d_1+d_2}} H_{t,\nu}(x,y) \, f(y) \, d\mu_{\nu}(y) \right|^{q} \, d\mu_{\nu}(x) \right)^{1/q} \\
& \nonumber \qquad \leq \int_{\R^{d_1+d_2}} |f(y)| \left( \int_{\R^{d_1+d_2}} V_{\nu}(x, t)^{\frac{q}{q'}} \, H_{t,\nu}(x,y)^{q} \, d\mu_{\nu}(x) \right)^{1/q} \, d\mu_{\nu}(y), \\ 
& \qquad \leq \sup_{y}  \left( \int_{\R^{d_1+d_2}} V_{\nu}(x, t)^{\frac{q}{q'}} \, H_{t,\nu}(x,y)^{q} \, d\mu_{\nu}(x) \right)^{1/q} \|f\|_{L^1(d\mu_{\nu})}. \label{eq:L^1_to_L^q_norm} 
\end{align}
But, for any $y \in \R^{d_1+d_2}$, we have 
\begin{align*}
& \int_{\R^{d_1+d_2}} V_{\nu}(x, t)^{\frac{q}{q'}} \, H_{t,\nu}(x,y)^{q} \, d\mu_{\nu}(x) \\ 
& \lesssim \int_{\R^{d_1+d_2}} V_{\nu}(x, t)^{\frac{q}{q'}} \left\{ e^{-qt|\nu|^2} \, e^{q \, \nu \cdot (x' - y')} \, V_{\nu}(x, t)^{-q} \, e^{-qb'\rho(x,y)^2/t} \right\} d\mu_{\nu}(x) \\
& \leq e^{-qt|\nu|^2} \int_{\rho(x,y) < \sqrt{t}} V_{\nu}(x, t)^{-1} \, e^{q \, \nu \cdot (x' - y')} \, e^{-qb'\rho(x,y)^2/t} \, d\mu_{\nu}(x) \\
& \quad + \sum_{j=0}^{\infty} e^{-qt|\nu|^2} \int_{2^{j}\sqrt{t} \, \leq \rho(x,y) < \, 2^{j+1} \sqrt{t}} V_{\nu}(x, t)^{-1} \, e^{q \, \nu \cdot (x' - y')} \, e^{-qb'\rho(x,y)^2/t} \, d\mu_{\nu}(x) \\
& =: I + \sum_{j=0}^{\infty} I_{j}.
\end{align*}
Estimation of $I$ is simple. In fact, 
\begin{align*}
I \lesssim e^{-qt|\nu|^{2}} \, \int_{\rho(x,y) < \sqrt{t}} |B(x, \sqrt{t})|^{-1} \, e^{|\nu|\sqrt{t}q} \, dx \lesssim \int_{\rho(x,y) < \sqrt{t}} |B(y, \sqrt{t})|^{-1} \, dx = 1.
\end{align*}
On the other hand, there exist $c' > 0$ such that for any $j \geq 0$, 
\begin{align*}
I_{j} 
& \lesssim e^{q |\nu| 2^{j+1} \sqrt{t}} \, e^{-qb'2^{2j}} \int_{\rho(x,y) < \, 2^{j+1} \sqrt{t}} |B(x, \sqrt{t})|^{-1} \, dx \\
& \lesssim_{C} e^{- c' 2^{2j}} \int_{\rho(x,y) < \, 2^{j+1} \sqrt{t}} \left(1+\frac{\rho(x,y)}{\sqrt{t}} \right)^{Q}|B(y, \sqrt{t})|^{-1} \, dx \\
& \leq e^{- c'2^{2j}} (1+2^{j+1})^{Q} \, |B(y, \sqrt{t} )|^{-1} \, |B(y, 2^{j+1} \sqrt{t})| \\
& \leq e^{- c'2^{2j}} (1+2^{j+1})^{2Q}.
\end{align*}
Summarising, the above estimates imply that 
\begin{align*}
\sup_{y} \int_{\R^{d_1+d_2}} V_{\nu}(x, t)^{\frac{q}{q'}} \, H_{t,\nu}(x,y)^{q} \, d\mu_{\nu}(x) \lesssim_{C} 1 + \sum_{j=0}^{\infty} e^{- c'2^{2j}} (1+2^{j+1})^{2Q} \lesssim 1, 
\end{align*}
and then putting this estimate in inequality \eqref{eq:L^1_to_L^q_norm}, we get 
\begin{align} \label{eq:L^1_to_L^q_norm_1}
\|V_{\nu}(\cdot, t)^{\frac{1}{q'}} \, e^{-tG_{\nu}}f\|_{L^{q}(d\mu_{\nu})} \lesssim_{C} \|f\|_{L^{1}(d\mu_{\nu})},  
\end{align}
which proves the claim of the proposition for $p = 1$ and $1 \leq q < \infty$. 

\medskip Next, we prove the proposition for $q = \infty$ and $1 \leq p < \infty$. For this, we consider 
\begin{align*}
|V_{\nu}(x, t)^{\frac{1}{p}} \, e^{-tG_{\nu}}f(x)| 
& \quad \leq \int_{\R^{d_1+d_2}} V_{\nu}(x, t)^{\frac{1}{p}} \, H_{t,\nu}(x,y) \, |f(y)| \, d\mu_{\nu}(y) \\
& \quad \leq \left(\int_{\R^{d_1+d_2}} V_{\nu}(x, t)^{\frac{p'}{p}} \, H_{t,\nu}(x,y)^{p'} \, d\mu_{\nu}(y)\right)^{1/p'} \, \|f\|_{L^{p}(d\mu_{\nu})} \\
& \quad \lesssim_{C} \|f\|_{L^{p}(d\mu_{\nu})},
\end{align*}
where in the last step we use
$$
\sup_{x} \int_{\R^{d_1+d_2}} V_{\nu}(x, t)^{\frac{p'}{p}} \, H_{t,\nu}(x,y)^{p'} \, d\mu_{\nu}(y) \lesssim_{C} 1,
$$
which can be proved using similar arguments as before. Thus, we get
\begin{align} \label{eq:L^p_to_L^infty}
\|V_{\nu}(\cdot, t)^{\frac{1}{p}} \, e^{-tG_{\nu}}f\|_{L^{\infty}(d\mu_{\nu})} \lesssim_{C} \|f\|_{L^{p}(d\mu_{\nu})}. 
\end{align}

To obtain the remaining cases $1 < p < q < \infty$, we would like to interpolate inequalities \eqref{eq:L^1_to_L^q_norm_1} and \eqref{eq:L^p_to_L^infty}. While doing so, we treat the ball volume factors $V_{\nu}(\cdot, t)^{\frac{1}{q'}}$ and $V_{\nu}(\cdot, t)^{\frac{1}{p}}$ as weight functions. But, since the left-hand side of \eqref{eq:L^p_to_L^infty} involves $L^{\infty}(d\mu_{\nu})$-norm, we do not know if there is an interpolation theorem that we can directly apply. To circumvent this situation, we can take help from Lemma \ref{lemma:ball_volume_and_heat_operator_commutation} and transfer these weight functions to the right-hand side of both the inequalities. More precisely, in view of Lemma \ref{lemma:ball_volume_and_heat_operator_commutation}, inequalities \eqref{eq:L^1_to_L^q_norm_1} and \eqref{eq:L^p_to_L^infty} are equivalent to 
\begin{align*}
& \|e^{-tG_{\nu}}f\|_{L^{r}(d\mu_{\nu})} \lesssim_{C} \|V_{\nu}(\cdot, t)^{-\frac{1}{r'}} \, f\|_{L^{1}(d\mu_{\nu})}, \\
\text{and} \qquad & \|e^{-tG_{\nu}}f\|_{L^{\infty}(d\mu_{\nu})} \lesssim_{C} \|V_{\nu}(\cdot, t)^{-\frac{1}{s}} \, f\|_{L^{s}(d\mu_{\nu})},
\end{align*}
for any $1 \leq r, s < \infty$. 

\medskip 
Now, given $1 < p < q < \infty$, denote by $1 < r < \infty$ the number satisfying $\frac{1}{r'} = \frac{1}{p} - \frac{1}{q}$. Let $\theta \in (0,1)$ be such that $\frac{1}{p} = \frac{1-\theta}{1} + \frac{\theta}{r'}$, so that $\frac{1}{q} = \frac{1}{p} - \frac{1}{r'} = \frac{1-\theta}{r} + \frac{\theta}{\infty}$. Then, one can make use of the interpolation \cite[\S 5.5]{Interpolation_spaces_book} to deduce that 
$$
\|e^{-tG_{\nu}}f\|_{L^{q}(d\mu_{\nu})} \lesssim_{C} \|V_{\nu}(\cdot, t)^{\frac{1}{q}-\frac{1}{p}} \, f\|_{L^{p}(d\mu_{\nu})},
$$
which is equivalent to (thanks once again to Lemma \ref{lemma:ball_volume_and_heat_operator_commutation}) 
$$
\|V_{\nu}(\cdot, t)^{\frac{1}{p}-\frac{1}{q}} \, e^{-tG_{\nu}}f\|_{L^{q}(d\mu_{\nu})} \lesssim_{C} \| f\|_{L^{p}(d\mu_{\nu})}.
$$
This completes the proof of Proposition \ref{prop:L^p_L^q_type_ineq_for_heat_operator}. 
\end{proof}


The vector-valued inequality in the following proposition will be frequently used in the proofs of our main results. We skip the proof of this one as it can easily be established following the arguments of the proof of \cite[Proposition 3.5]{Bruno_Marco_Vallarino_Besov_TL_20}. 

\begin{proposition} \label{prop:semigroup_version_of_Fefferman_Stein_theorem}
Let $1 < p < \infty, \, 1 \leq q \leq \infty$ and $0 < \kappa' < \kappa$. Then, for any sequence of measurable functions $(t_j)$ such that $t_j(x) \in [\kappa' 2^{-j}, \kappa 2^{-j}]$ for all $j \in \N_{0}$ and $x \in \R^{d_1+d_2}$, we have 
$$
\left\|\left(\sum_{j \in \N_{0}} |e^{-t_{j}(.) G_{\nu}} f_{j}|^{q} \right)^{1/q}\right\|_{L^{p}(d\mu_{\nu})} \lesssim \left\|\left(\sum_{j \in \N_{0}} |f_{j}|^{q} \right)^{1/q}\right\|_{L^{p}(d\mu_{\nu})}, \qquad \text{when } \, q < \infty,
$$
and when $q = \infty$, 
$$
\left\|\sup_{j \in \N_{0}} |e^{-t_{j}(.) G_{\nu}} f_{j}|\right\|_{L^{p}(d\mu_{\nu})} \lesssim \left\|\sup_{j \in \N_{0}} |f_{j}|\right\|_{L^{p}(d\mu_{\nu})}, 
$$
for every sequence $(f_j)$ of measurable functions in $\mathcal{S}'(\R^{d_1+d_2})$. 
\end{proposition}

\medskip 
Next, following \cite[Proposition 3.6]{Bruno_Marco_Vallarino_Besov_TL_20}, we establish a continuous version of the above proposition with the involvement of the ball volume in an appropriate sense.
\begin{proposition} \label{prop:semigroup_continuous_version_of_Fefferman_Stein_theorem}
Let $1 < p < \infty, \, 1 \leq q \leq \infty, \, \alpha \geq 0, \, c > 0$ and $\eta \in \R$. Then, 
\begin{align*}
& \left\|\left(\int_{0}^1 \left( V_{\nu}(\cdot, t)^{-\alpha/Q}\, V_{\nu}(\cdot, 1)^{\eta} \, e^{-ctG_\nu} |F(t,\cdot)| \right)^{q} \frac{dt}{t} \right)^{1/q} \right\|_{L^p(d\mu_{\nu})} \\
& \quad \lesssim  \left\|\left(\int_{0}^1 \left(V_{\nu}(\cdot, t)^{-\alpha/Q}\, V_{\nu}(\cdot, 1)^{\eta} \, |F(t,\cdot)|\right)^q \frac{dt}{t} \right)^{1/q} \right\|_{L^p(d\mu_{\nu})},   
\end{align*}
for any measurable function $F$ such that $F(t, \cdot) \in \mathcal{S}'(\R^{d_1+d_2})$ for each $t \in (0,1)$, and with obvious modifications when $q = \infty$. 
\end{proposition}
\begin{proof}
We shall prove this proposition only in the case when $1 \leq q < \infty$. Using Lemma \ref{lem:heat_operator_and_derivative} (with $\kappa' = c, \, \kappa = 2c$ and $t_0 = 2^{-j}$), we have
\begin{align*}
& \int_{0}^1 \left( V_{\nu}(\cdot, t)^{-\alpha/Q}\, V_{\nu}(\cdot, 1)^{\eta} \, e^{-ctG_\nu} |F(t,\cdot)| \right)^{q} \frac{dt}{t} \\
& \lesssim \sum_{j=1}^{\infty} \int_{2^{-j}}^{2^{-j+1}} \left(V_{\nu}(\cdot, 2^{-j})^{-\alpha/Q}\, V_{\nu}(\cdot, 1)^{\eta} \, e^{-cb 2^{-j+1}G_\nu} |F(t,\cdot)| \right)^{q} \frac{dt}{t} \\  
& \lesssim \sum_{j=1}^{\infty} \int_{2^{-j}}^{2^{-j+1}} \left(e^{-4cb^{3} 2^{-j+1}G_\nu}\left(V_{\nu}(\cdot, 2^{-j})^{-\alpha/Q}\, V_{\nu}(\cdot, 1)^{\eta} \,  |F(t,\cdot)|\right) \right)^{q} \frac{dt}{t} \\
& \lesssim \sum_{j=1}^{\infty} \left(e^{-4cb^{3} 2^{-j+1}G_\nu} \left( \int_{2^{-j}}^{2^{-j+1}} \left(V_{\nu}(\cdot, 2^{-j})^{-\alpha/Q}\, V_{\nu}(\cdot, 1)^{\eta} \, |F(t,\cdot)| \right)^{q} \frac{dt}{t} \right)^{1/q} \right)^{q}, 
\end{align*} 
where the last step follows from Minkowski's inequality applied in the integral representation of the heat semigroup $e^{-4cb^{3} 2^{-j+1}G_\nu}$. 

\medskip 
With the above estimate, we can use Proposition \ref{prop:semigroup_version_of_Fefferman_Stein_theorem} to conclude that 
\begin{align*}
& \left\| \left(\int_{0}^1 \left(V_{\nu}(\cdot, t)^{-\alpha/Q} \, V_{\nu}(\cdot, 1)^{\eta} \, e^{-ctG_\nu} |F(t,\cdot)| \right)^{q} \frac{dt}{t} \right)^{1/q} \right\|_{L^p(d\mu_{\nu})} \\ 
& \lesssim \left\|\left(\sum_{j=1}^{\infty} \int_{2^{-j}}^{2^{-j+1}} \left(V_{\nu}(\cdot, 2^{-j})^{-\alpha/Q}\, V_{\nu}(\cdot, 1)^{\eta} \, |F(t,\cdot)| \right)^{q} \frac{dt}{t} \right)^{1/q} \right\|_{L^p(d\mu_{\nu})} \\
& \lesssim \left\|\left(\sum_{j=1}^{\infty} \int_{2^{-j}}^{2^{-j+1}} \left(V_{\nu}(\cdot, t)^{-\alpha/Q}\, V_{\nu}(\cdot, 1)^{\eta} \, |F(t,\cdot)| \right)^{q} \frac{dt}{t} \right)^{1/q} \right\|_{L^p(d\mu_{\nu})} \\
& = \left\|\left( \int_{0}^{1} \left( V_{\nu}(\cdot, t)^{-\alpha/Q}\, V_{\nu}(\cdot, 1)^{\eta} \, |F(t,\cdot)| \right)^{q} \frac{dt}{t} \right)^{1/q} \right\|_{L^p(d\mu_{\nu})},
\end{align*} 
and this completes the proof of the proposition.
\end{proof}

Finally, we record the following application of Schur's lemma. For proof, we refer to \cite[Lemma 2.2]{Feneuil_Joseph_Algebra_properties_2018} and \cite[Lemma 4.3]{Bruno_Marco_Vallarino_Besov_TL_20}.
\begin{lemma} \label{lem:Cor_of_Schur_lemma}
Let $\gamma, \eta \in \R$ be such that $0 < \gamma < \eta$, and let $q \in [1, \infty]$. 
\begin{enumerate}[(i)]
\item Let $a, \, b \in \mathbb{Z} \cup \{\pm \infty\}$ be such that $a < b$, then for any sequence $(r_{n})_{n \in \mathbb{Z}} \subset [0, \infty)$, we have
$$
\sum_{j=a}^{b} \left(2^{-j\gamma} \sum_{n=a}^{b} 2^{\min\{n,j\} \eta} r_{n} \right)^{q} \lesssim \sum_{n=a}^{b} (2^{(-\gamma+ \eta)n } \, r_{n})^{q},
$$
with obvious modifications when $q = \infty$.

\medskip 
\item Let $r:(0,1) \to [0, \infty)$ be any function, then
$$
\int_{0}^{1} \left(u^{\gamma} \int_{0}^{1} \frac{1}{(t+u)^{\eta}} r(t) \, \frac{dt}{t} \right)^{q} \frac{du}{u} \lesssim \int_{0}^{1} (t^{\gamma- \eta} r(t))^{q} \, \frac{dt}{t},
$$
with obvious modifications when $q=\infty$.
\end{enumerate}
\end{lemma}


\section{Some characterisations of Besov and Triebel--Lizorkin spaces} 
\label{sec:spaces-definition-and-characterisation}

In this section, we shall prove several characterisations of the Besov and Triebel--Lizorkin norms. In various of our arguments, we shall make use of the following fundamental decomposition formula whose proof can be written by just repeating the proof of \cite[Lemma 3.1]{Feneuil_Joseph_Algebra_properties_2018}, so we omit the details. 
\begin{lemma} \label{lem:decomposition_of_a_distibution/function}
Let $m \in \N$. For any $\phi \in \mathcal{S}(\R^{d_1 + d_2})$ and $f \in \mathcal{S}'(\R^{d_1 + d_2})$, we have 
\begin{align*}
\phi & = \frac{1}{(m-1)!}\int_{0}^{1} W_{t}^{(m)} \phi \, \frac{dt}{t} + \sum_{k=0}^{m-1} \frac{1}{k!} \, W_{1}^{(k)} \phi, \\  
\text{and} \qquad f & = \frac{1}{(m-1)!}\int_{0}^{1} W_{t}^{(m)} f \, \frac{dt}{t} + \sum_{k=0}^{m-1} \frac{1}{k!} \, W_{1}^{(k)} f,
\end{align*}
with the integrals converging in $\mathcal{S}(\R^{d_1 + d_2})$ and $\mathcal{S}'(\R^{d_1 + d_2})$ respectively. 
\end{lemma}

\subsection{Independence of parameters}
\label{subsec:independence_of_parameters}
We have the following results concerning the independence of certain parameters in the definitions of Besov and Triebel--Lizorkin norms. 
\begin{theorem} \label{thm:independence_of_parameters_Besov}
Let $1 \leq p, \, q \leq \infty, \, \eta \in \mathbb{R}$ and $\alpha > 0$. 
Then, the classical Besov norm $\|f\|_{B^{p, q}_{\alpha, \eta}(d\mu_{\nu})}$ is equivalent to 
\begin{equation*} 
\left(\int_{0}^{1} \|t^{-\alpha/2} \, V_{\nu}^{\eta} \, |W_{t}^{(m)} f|\|^{q}_{L^{p}(d\mu_{\nu})} \, \frac{dt}{t} \right)^{1/q} + \|V_{\nu}^{\eta} \, e^{-t_0 G_{\nu}} f\|_{L^{p}(d\mu_{\nu})}, 
\end{equation*} 
and the non-classical Besov norm $\|f\|_{\tilde{B}^{p,q}_{\alpha,\eta}(d\mu_{\nu})}$ is equivalent to 
\begin{equation*} 
\left(\int_{0}^{1} \|V_{\nu}(t)^{-\alpha/Q} \, V_{\nu}^{\eta} \, |W_{t}^{(m)} f|\|^{q}_{L^{p}(d\mu_{\nu})} \, \frac{dt}{t} \right)^{1/q} + \|V_{\nu}^{-\alpha/Q+\eta} e^{-t_0 G_{\nu}} f\|_{L^{p}(d\mu_{\nu})}, 
\end{equation*} 
for any $0 \leq t_0 < 1$ and $m \in \mathbb{N}$ such that $m > \alpha/2$, with appropriate modifications for $q = \infty$. Also, when $\alpha = 0$, same equivalence holds with $t_0 \in (0,1)$.
\end{theorem}

\begin{theorem} \label{thm:independence_of_parameters_Triebel_Lizorkin}
Let $1 < p < \infty, \, 1 \leq  q \leq \infty, \, \eta \in \mathbb{R}$ and $\alpha > 0$. Then, the classical Triebel--Lizorkin norm $\|f\|_{F^{p, q}_{\alpha, \eta}(d\mu_{\nu})}$ is equivalent to 
\begin{equation*} 
\left\|\left(\int_{0}^{1} |t^{-\alpha/2} \, V_{\nu}^{\eta} \, W_{t}^{(m)} f|^{q} \, \frac{dt}{t} \right)^{1/q} \right\|_{L^{p}(d\mu_{\nu})} + \|V_{\nu}^{\eta} \, e^{-t_0 G_{\nu}} f\|_{L^{p}(d\mu_{\nu})}, 
\end{equation*}
and the non-classical Triebel--Lizorkin norm $\|f\|_{\tilde{F}^{p,q}_{\alpha,\eta}(d\mu_{\nu})}$ is equivalent to 
\begin{equation*} 
\left\|\left(\int_{0}^{1} |V_{\nu}(t)^{-\alpha/Q} \, V_{\nu}^{\eta} \, W_{t}^{(m)} f|^{q} \, \frac{dt}{t} \right)^{1/q} \right\|_{L^{p}(d\mu_{\nu})} + \|V_{\nu}^{-\alpha/Q+\eta} e^{-t_0 G_{\nu}} f\|_{L^{p}(d\mu_{\nu})}, 
\end{equation*} 
for any $0 \leq t_0 < 1$ and $m \in \N$ such that $m > \alpha/2$, with appropriate modifications for $q = \infty$. Also, when $\alpha = 0$, same equivalence holds with $t_0 \in (0,1)$.   
\end{theorem}

We shall only prove Theorem \ref{thm:independence_of_parameters_Triebel_Lizorkin} as Theorem \ref{thm:independence_of_parameters_Besov} can be proved following the same ideas. In fact, the computations in the proof of Theorem \ref{thm:independence_of_parameters_Besov} are less involved than those in Theorem \ref{thm:independence_of_parameters_Triebel_Lizorkin}. 

\begin{proof} We shall only prove the non-classical case as the classical case is similar. Also, we shall write the detailed arguments only for $1 \leq q < \infty$, and it will become clear from the proof that the treatment of $q = \infty$ case is the same, with only the routine changes. The proof will be done in three steps. 

\medskip 
\textbf{\underline{Step 1}:} In this step, we record the following estimate: 
\begin{align*}
& \left\|\left(\int_{0}^{1} (V_{\nu}(t)^{-\alpha/Q} \, V_{\nu}^{\eta} \, |W_t^{(m+1)}f|)^{q} \, \frac{dt}{t}\right)^{\frac{1}{q}} \right\|_{L^p(d\mu_{\nu})} \\
& \lesssim \left\|\left( \int_{0}^{1} (V_{\nu}(t)^{-\alpha/Q} \, V_{\nu}^{\eta} \, |W_t^{(m)}f|)^{q} \, \frac{dt}{t}\right)^{\frac{1}{q}} \right\|_{L^p(d\mu_{\nu})}, 
\end{align*}
which follows immediately from Proposition \ref{prop:semigroup_continuous_version_of_Fefferman_Stein_theorem} once we have $|W_t^{(m+1)}f| \lesssim e^{-\frac{b}{2} tG_\nu} |W_{t/2}^{(m)} f|$ and the latter can be shown to hold true with the help of Lemma \ref{lem:heat_operator_and_derivative} as follows:  
\begin{align*}
|W_t^{(m+1)}f| = |(tG_\nu)^{m+1} \, e^{-t G_\nu} f| 
\lesssim |(t G_\nu) \, e^{-\frac{t}{2} G_\nu}(W_{t/2}^{(m)} f)| \lesssim e^{-\frac{b}{2} tG_\nu} |W_{t/2}^{(m)} f|. 
\end{align*}

\medskip 
\textbf{\underline{Step 2}:} In this step, we shall show that  
\begin{align*}
& \left\|\left( \int_{0}^{1} (V_{\nu}(t)^{-\alpha/Q} \, V_{\nu}^{\eta} \, |W_t^{(m)}f|)^q \, \frac{dt}{t}\right)^{\frac{1}{q}} \right\|_{L^p(d\mu_{\nu})} \\
& \lesssim \left\|\left( \int_{0}^{1} (V_{\nu}(t)^{-\alpha/Q} \, V_{\nu}^{\eta} \,  |W_t^{(m+1)}f|)^q \, \frac{dt}{t}\right)^{\frac{1}{q}} \right\|_{L^p(d\mu_{\nu})} + \|V_{\nu}^{-\alpha/Q + \eta} e^{-t_0 G_\nu}f\|_{L^p(d\mu_{\nu})}, 
\end{align*}
for any $\alpha \geq 0$ and $0 \leq t_0 < 1$. 

\medskip For the same, note first that it follows from integration by parts that 
$$
W_t^{(m)} f = W_t^{(m)} e^{-G_\nu} f + \int_{0}^1 t^{-1} W_t^{(m+1)} e^{-s G_\nu} f \, ds,
$$
using which we have 
\begin{align*}
&\left\|\left( \int_{0}^{1} (V_{\nu}(t)^{-\alpha/Q} \, V_{\nu}^{\eta} \, |W_t^{(m)} f|)^q \, \frac{dt}{t}\right)^{\frac{1}{q}} \right\|_{L^p(d\mu_{\nu})}  \\
& \lesssim \left\|\left( \int_{0}^{1} (V_{\nu}(t)^{-\alpha/Q}\, V_{\nu}^{\eta} \, |W_t^{(m)} e^{-G_\nu} f|)^q \, \frac{dt}{t}\right)^{\frac{1}{q}} \right\|_{L^p(d\mu_{\nu})}   \\
& \quad + \left\|\left( \int_{0}^{1} \left(V_{\nu}(t)^{-\alpha/Q}\, V_{\nu}^{\eta} \,  \left|\int_{0}^1 t^{-1} W_t^{(m+1)} e^{-s G_\nu} f \, ds \right|\right)^q \, \frac{dt}{t}\right)^{\frac{1}{q}} \right\|_{L^p(d\mu_{\nu})} \\
& =: I_1 + I_2. 
\end{align*}
First, we shall estimate $I_1$. For that, we consider 
\begin{align*}
|W_t^{(m)} e^{-G_\nu}f| & = |(tG_\nu)^m \, e^{-(t+1-t_0)G_\nu} \, (e^{-t_0 G_\nu} f)| \\ 
& \lesssim t^m \, (t+1-t_0)^{-m} \, e^{-b \, (t+1-t_0) G_\nu} \, |e^{-t_0 G_\nu}f| \\
& \lesssim e^{-b \, (1-t_0)G_\nu} \left( t^m \, e^{-b t G_\nu} \, |e^{-t_0G_\nu}f| \right),
\end{align*}
where the first inequality is due to Lemma \ref{lem:heat_operator_and_derivative}.
Now, using the above estimate, let us perform the following computations using Lemma \ref{lemma:ball_volume_and_heat_operator_commutation}, Proposition \ref{prop:semigroup_continuous_version_of_Fefferman_Stein_theorem}, and the Minkowski's integral inequality at intermediary steps to have 
\begin{align*}
I_1 & \lesssim \left\|\left( \int_{0}^{1} \left( V_{\nu}(t)^{-\alpha/Q}\, V_{\nu}^{\eta} \, e^{-b \, (1-t_0) G_\nu} \left( t^m \, e^{-b \, tG_\nu} \, |e^{-t_0 G_\nu} f| \right) \right)^q \, \frac{dt}{t}\right)^{\frac{1}{q}} \right\|_{L^p(d\mu_{\nu})} \\ 
& \lesssim \left\|e^{-2b^2 (1-t_0)G_{\nu}}\left(\int_{0}^{1} (V_{\nu}^{-\alpha/Q + \eta} \,  t^{m-\alpha/2} \, e^{-b \, tG_\nu} \, |e^{-t_0 G_\nu} f|)^q \, \frac{dt}{t}\right)^{\frac{1}{q}} \right\|_{L^p(d\mu_{\nu})} \\
& \lesssim \left\|\left(\int_{0}^{1} (V_{\nu}^{-\alpha/Q + \eta} \, t^{m-\alpha/2} \, e^{-btG_{\nu}} \, |e^{-t_0 G_\nu} f|)^q \, \frac{dt}{t}\right)^{\frac{1}{q}} \right\|_{L^p(d\mu_{\nu})} \\ 
& \lesssim \left\|\left(\int_{0}^{1} (V_{\nu}^{-\alpha/Q + \eta} \, t^{m-\alpha/2} \, |e^{-t_0 G_\nu} f|)^q \, \frac{dt}{t}\right)^{\frac{1}{q}} \right\|_{L^p(d\mu_{\nu})} \\ 
& \lesssim \left\|V_{\nu}^{-\alpha/Q  + \eta} e^{-t_0 G_\nu}f\right\|_{L^p(d\mu_{\nu})}. 
\end{align*}

\medskip Next, we consider $I_2$ and split the integral in $s$-variable into two parts, namely on the intervals $(0,t)$ and $[t,1]$. Let us first consider the integral concerning $s \in (0,t)$. For the same, note that
\begin{align*}
|W_{t}^{(m+1)}\, e^{-s G_\nu} f|  \lesssim |e^{-(s+t/2) G_\nu} W_{t/2}^{(m+1)} f| \lesssim  e^{-\frac{3}{2} b tG_\nu} |W_{t/2}^{(m+1)} f|,
\end{align*}
where we used Lemma \ref{lem:heat_operator_and_derivative} with $\kappa' = \frac{1}{2}$ and $\kappa = \frac{3}{2}$. The above estimate along with Proposition \ref{prop:semigroup_continuous_version_of_Fefferman_Stein_theorem} implies that 
\begin{align*}
& \left\|\left(\int_{0}^{1}  \left(V_{\nu}(t)^{-\alpha/Q} \,  V_{\nu}^{\eta} \, t^{-1} \int_{0}^{t} |W_t^{(m+1)} \, e^{-s G_\nu} f| \, ds \right)^{q} \frac{dt}{t} \right)^{1/q} \right\|_{L^p(d\mu_{\nu})} \\ 
& \quad \lesssim \left\|\left(\int_{0}^{1} \left(V_{\nu}(t)^{-\alpha/Q} \, V_{\nu}^{\eta} \, |W_{t}^{(m+1)} f| \right)^{q} \frac{dt}{t}\right)^{1/q}\right\|_{L^p(d\mu_{\nu})}.
\end{align*}

On the other hand, using Proposition \ref{prop:semigroup_continuous_version_of_Fefferman_Stein_theorem} in the other part involving the integral over $s \in [t,1]$, we have 
\begin{align*}
&\left\|\left(\int_{0}^{1}  \left(V_{\nu}(t)^{-\alpha/Q} \, V_{\nu}^{\eta} \, t^{-1} \int_{t}^{1} |W_t^{(m+1)} \, e^{-s G_\nu}f| \, ds \right)^{q} \frac{dt}{t} \right)^{1/q} \right\|_{L^p(d\mu_{\nu})} \\ 
& \leq \left\| \left(\int_{0}^{1}  \left(V_{\nu}(t)^{-\alpha/Q} \, V_{\nu}^{\eta} \, t^{m} \, e^{-t G_\nu} \int_{t}^{1} s^{-(m+1)} |W_s^{(m+1)} f| \, ds \right)^{q} \frac{dt}{t} \right)^{1/q}\right\|_{L^p(d\mu_{\nu})} \\
& \leq \left\| \left(\int_{0}^{1}  \left(V_{\nu}(t)^{-\alpha/Q} \, V_{\nu}^{\eta} \, t^{m} \int_{t}^{1} s^{-(m+1)}|W_s^{(m+1)} f| \, ds \right)^{q} \frac{dt}{t} \right)^{1/q}\right\|_{L^p(d\mu_{\nu})}\\ & \leq \left\| \left(\int_{0}^{1}  \left(\int_{0}^{1} K(s,t) \, g(s) \, \frac{ds}{s}\right)^{q}\frac{dt}{t} \right)^{1/q}\right\|_{L^p(d\mu_{\nu})},
\end{align*}
where 
\begin{align*}
K(s,t) & = \left(\frac{t}{s}\right)^{m} \left(\frac{V_{\nu}(s)}{V_{\nu}(t)}\right)^{\alpha/Q} \chi_{\{s\geq t\}}, \\ 
\text{and} \qquad \qquad 
g(s) & = V_{\nu}(s)^{-\alpha/Q } \, V_{\nu}^{\eta} \, |W_s^{(m+1)}f|.
\end{align*}
Now, it is easy to prove that
$$
\sup_{t \in (0,1)} \int_{0}^{1} K(s,t) \, \frac{ds}{s} \lesssim 1 \qquad \text{and} \qquad \sup_{s \in (0,1)} \int_{0}^{1} K(s,t) \, \frac{dt}{t} \lesssim 1, 
$$
and with this, one can invoke Schur's lemma to conclude that 
\begin{align*}
&\left\|\left(\int_{0}^{1}  \left(V_{\nu}(t)^{-\alpha/Q} \, V_{\nu}^{\eta} \, t^{-1} \int_{t}^{1} |W_t^{(m+1)} e^{-s G_\nu} f| \, ds \right)^{q} \frac{dt}{t} \right)^{1/q}\right\|_{L^p(d\mu_{\nu})} \\ 
& \lesssim \left\| \left(\int_{0}^{1} g(t)^{q} \frac{dt}{t} \right)^{1/q}\right\|_{L^p(d\mu_{\nu})} \\
& = \left\| \left(\int_{0}^{1} (V_{\nu}(t)^{-\alpha/Q} \, V_{\nu}^{\eta} \, |W_t^{(m+1)} f|)^{q} \frac{dt}{t} \right)^{1/q}\right\|_{L^p(d\mu_{\nu})},
\end{align*}
which completes the proof of Step 2.

\medskip 
\textbf{\underline{Step 3}:} In this final step, we shall prove the following estimate: 
\begin{align*}
& \|V_{\nu}^{-\alpha/Q + \eta} \, e^{-t_0G_{\nu}} f\|_{L^p(d\mu_{\nu})} \\
& \lesssim \left\|\left( \int_{0}^{1} (V_{\nu}(t)^{-\alpha/Q}\, V_{\nu}^{\eta} \, |W_t^{(m)}f|)^q \, \frac{dt}{t}\right)^{\frac{1}{q}} \right\|_{L^p(d\mu_{\nu})} + \|V_{\nu}^{-\alpha/Q + \eta} \, e^{-\frac{1}{2}G_{\nu}} f\|_{L^p(d\mu_{\nu})}, 
\end{align*}
and it will be clear from the proof of the above estimate that one can perform similar computations to also prove the following analogous estimate reversing the roles of exponent $t_0$ and $1/2$: 
\begin{align*}
& \|V_{\nu}^{-\alpha/Q + \eta} \, e^{-\frac{1}{2}G_{\nu}} f\|_{L^p(d\mu_{\nu})} \\
& \lesssim \left\|\left( \int_{0}^{1} (V_{\nu}(t)^{-\alpha/Q}\, V_{\nu}^{\eta} \, |W_t^{(m)}f|)^q \, \frac{dt}{t}\right)^{\frac{1}{q}} \right\|_{L^p(d\mu_{\nu})} + \|V_{\nu}^{-\alpha/Q + \eta} \, e^{-t_{0}G_{\nu}} f\|_{L^p(d\mu_{\nu})}. \end{align*}

Let us first assume that $\alpha > 0$ and $0 \leq t_0 < 1$. In this case, note that by making use of Lemma \ref{lem:decomposition_of_a_distibution/function}, we can write 
\begin{align*}
\|V_{\nu}^{-\alpha/Q + \eta} \, e^{-t_0G_{\nu}} f\|_{L^p(d\mu_{\nu})} & \lesssim 
\left\|V_{\nu}^{-\alpha/Q + \eta} \int_{0}^{1} W_t^{(m)}(e^{-t_0G_{\nu}}f) \, \frac{dt}{t} \right\|_{L^p(d\mu_{\nu})} \\ 
& \quad + \sum_{k=0}^{m-1} \|V_{\nu}^{-\alpha/Q + \eta} \, W_1^{(k)}(e^{-t_0G_{\nu}} f)\|_{L^p(d\mu_{\nu})}
\\ & =: I_3 +I_4. 
\end{align*}
Working with the term $I_4$, it follows from Lemma \ref{lem:heat_operator_and_derivative} that for any $k \geq 0$, 
\begin{align*}
|W_1^{(k)} \, e^{-t_0G_{\nu}} f(x)| & \lesssim (t_0+1/2)^{-k} e^{-b \, (t_0+1/2) G_\nu} \, |e^{-\frac{1}{2}G_\nu}f|(x), 
\end{align*}
and then by Lemma \ref{lemma:ball_volume_and_heat_operator_commutation} we get 
\begin{align*}
I_4 \lesssim
\| e^{-2b^2 \, (t_0+1/2)G_\nu} (V_{\nu}^{-\alpha/Q + \eta} |e^{-\frac{1}{2}G_\nu}f|)\|_{L^p(d\mu_{\nu})} \lesssim
\|V_{\nu}^{-\alpha/Q + \eta} e^{-\frac{1}{2}G_\nu}f\|_{L^p(d\mu_{\nu})}, 
\end{align*}
where the last step follows from the $L^p(d\mu_{\nu})$-boundedness of the heat operator $e^{-t G_\nu}$ with bound uniform in $t>0$. 

\medskip 
Next, we consider $I_3$. For the same, note first that 
\begin{align*}
V_{\nu}(x, 1)^{-\alpha/Q + \eta} \int_{0}^{1} |W_t^{(m)}(e^{-t_0G_{\nu}}f)(x)|  \, \frac{dt}{t} \lesssim e^{- 2b \, t_0 G_\nu} \left( V_{\nu}^{-\alpha/Q + \eta}
\int_{0}^{1} |W_t^{(m)} f| \frac{\,dt}{t} \right)(x),
\end{align*}
and therefore,  
\begin{align*}
I_3 & \lesssim \left\| V_{\nu}^{-\alpha/Q + \eta}
\int_{0}^{1} |W_t^{(m)} f| \frac{\,dt}{t} \right\|_{L^p(d\mu_{\nu})} \\ 
& = \left\|
\int_{0}^{1} \left(\frac{V_{\nu}(t)}{V_{\nu}}\right)^{\alpha/Q} V_{\nu}(t)^{-\alpha/Q} V_{\nu}^{\eta} |W_t^{(m)} f| \frac{\,dt}{t} \right\|_{L^p(d\mu_{\nu})}.
\end{align*}
Since $\frac{V_{\nu}(t)}{V_{\nu}} \leq t^{\frac{d}{2}}$ and $\alpha > 0$, a simple application of H\"{o}lder's inequality implies the claimed estimate 
\begin{align*}
I_3 & \lesssim \left\|
\left(\int_{0}^{1} (V_{\nu}(t)^{-\alpha/Q} \, V_{\nu}^{\eta} \, |W_t^{(m)} f|)^q \frac{\,dt}{t} \right)^{1/q} \right\|_{L^p(d\mu_{\nu})}.
\end{align*}

When $\alpha = 0$ and $t_0 \in (0,1)$, we may have to tweak the above argument as follows. First of all, we would use $m+1$ instead of $m$ in the decomposition formula leading to the terms $I_3$ and $I_4$. The treatment of $I_4$ remains unchanged. But, for $I_3$, using Lemmas \ref{lem:heat_operator_and_derivative} and \ref{lemma:ball_volume_and_heat_operator_commutation}, we get  
\begin{align*}
I_{3} & = \left\|V_{\nu}^{\eta} \int_{0}^{1} W_t^{(m+1)}(e^{-t_0G_{\nu}}f) \, \frac{dt}{t} \right\|_{L^p(d\mu_{\nu})} \\
& \lesssim \left\|V_{\nu}^{\eta} \int_{0}^{1} t \, e^{-bt_0G_{\nu}} \, |W_t^{(m)}f| \, \frac{dt}{t} \right\|_{L^p(d\mu_{\nu})} \\
& \lesssim \left\|\int_{0}^{1} t \, V_{\nu}^{\eta} \, |W_t^{(m)}f| \, \frac{dt}{t} \right\|_{L^p(d\mu_{\nu})} 
\lesssim \left\|\left(\int_{0}^{1} ( V_{\nu}^{\eta} \, |W_t^{(m)}f|)^{q} \, \frac{dt}{t} \right)^{1/q} \right\|_{L^p(d\mu_{\nu})}, 
\end{align*}
where the last step follows from the H\"{o}lder's inequality, and this completes the proof of Step 3.

\medskip 
These three steps together complete the proof of the independence of parameters.
\end{proof}


\subsection{Littlewood--Paley type characterisation} 
\label{subsec:Littlewood_Paley_characterisation}
Here, we shall prove a Littlewood--Paley type characterisation of the Besov and Triebel--Lizorkin norms. 
\begin{theorem} \label{thm:discrete-characterisation-Besov}
Let $1 \leq p, \, q \leq \infty, \, \alpha > 0$ and $\eta \in \R$. Then, the classical Besov norm $\displaystyle \|f\|_{B^{p,  q}_{\alpha, \eta}(d\mu_{\nu})}$ is equivalent to 
\begin{equation}
\left(\sum_{j=1}^{\infty} \|2^{j\alpha/2} \, V_{\nu}^{\eta} \, |W_{2^{-j}}^{(m)} f|\|_{L^p(d\mu_{\nu})}^q \right)^{1/q} + \|V_{\nu}^{\eta} \, e^{-t_0 G_{\nu}} f\|_{L^{p}(d\mu_{\nu})}, 
\end{equation}
and the non-classical Besov norm $\displaystyle \|f\|_{\tilde{B}^{p,  q}_{\alpha, \eta}(d\mu_{\nu})}$ is equivalent to
\begin{equation}
\left(\sum_{j=1}^{\infty} \|V_{\nu}(2^{-j})^{-\alpha/Q} \, V_{\nu}^{\eta} \, |W_{2^{-j}}^{(m)} f|\|_{L^p(d\mu_{\nu})}^q \right)^{1/q} + \|V_{\nu}^{-\alpha/Q +\eta} e^{-t_0 G_{\nu}} f\|_{L^{p}(d\mu_{\nu})}, 
\end{equation}
for any $0 \leq t_0 < 1$ and $m \in \N$ such that $m > \alpha/2$, with the usual modifications for $q = \infty$.
Also, if $\alpha = 0$, then above result holds for $t_0 \in (0,1)$.
\end{theorem}

\begin{theorem} \label{thm:discrete_characterisation-Triebel--Lizorkin} 
Let $1 < p < \infty, \, 1 \leq q \leq \infty, \, \alpha > 0$ and  $\eta \in \R$. Then, the classical Triebel--Lizorkin norm $\displaystyle \|f\|_{F^{p,  q}_{\alpha, \eta}(d\mu_{\nu})}$ is equivalent to 
\begin{equation}
\left\|\left(\sum_{j=1}^{\infty} |2^{j\alpha/2}\, V_{\nu}^{\eta} \, W_{2^{-j}}^{(m)} f|^{q} \right)^{1/q}\right\|_{L^p(d\mu_{\nu})} + \|V_{\nu}^{\eta} \, e^{-t_0 G_{\nu}} f\|_{L^{p}(d\mu_{\nu})},    
\end{equation}
and the non-classical Triebel--Lizorkin norm $\displaystyle \|f\|_{\tilde{F}^{p,  q}_{\alpha, \eta}(d\mu_{\nu})}$ is equivalent to 
\begin{equation}
\left\|\left(\sum_{j=1}^{\infty} |V_{\nu}(2^{-j})^{-\alpha/Q} \, V_{\nu}^{\eta} \, W_{2^{-j}}^{(m)} f|^{q} \right)^{1/q}\right\|_{L^p(d\mu_{\nu})} + \|V_{\nu}^{-\alpha/Q +\eta} e^{-t_0 G_{\nu}} f\|_{L^{p}(d\mu_{\nu})},    
\end{equation}
for any $0 \leq t_0 < 1$ and $m \in \N$ such that $m > \alpha/2$, with the usual modifications for $q = \infty$.
Also, if $\alpha = 0$, then above result holds for $t_0 \in (0,1)$.
\end{theorem}

We shall only prove Theorem \ref{thm:discrete_characterisation-Triebel--Lizorkin}. The proof of Theorem \ref{thm:discrete-characterisation-Besov} can be verified following similar ideas. 

\begin{proof}[Proof of Theorem \ref{thm:discrete_characterisation-Triebel--Lizorkin}] 
We shall only prove the non-classical case. The classical case can be proved by following similar steps and with lesser difficulty. In view of Theorem \ref{thm:independence_of_parameters_Triebel_Lizorkin}, it suffices to prove that 
\begin{align} 
\label{main-ineq-discrete_characterisation}
&\left\|\left(\int_{0}^{1}(V_{\nu}(t)^{-\alpha/Q}\, V_{\nu}^{\eta} \, |W_t^{(m)} f|)^q \, \frac{dt}{t} \right)^{1/q} \right\|_{L^p(d\mu_{\nu})} \\
\nonumber & \sim \left\|\left(\sum_{j=1}^{\infty}(V_{\nu}(2^{-j})^{-\alpha/Q} \, V_{\nu}^{\eta} \, |W_{2^{-j}}^{(m)} f|)^q \right)^{1/q} \right\|_{L^p(d\mu_{\nu})}.
\end{align}

For $\lesssim$ part in \eqref{main-ineq-discrete_characterisation}, note that 
\begin{align*}
&\int_{0}^{1}(V_{\nu}(t)^{-\alpha/Q}\, V_{\nu}^{\eta} \, |W_t^{(m)} f|)^q \, \frac{dt}{t} \\
& \lesssim \sum_{j=1}^\infty \int_{2^{-j}}^{2^{-j+1}}  (V_{\nu}(2^{-j})^{-\alpha/Q}  \, V_{\nu}^{\eta} \, |(2^{-j}G_\nu)^m e^{-(t-2^{-j-1})G_\nu} (e^{-2^{-j-1}G_\nu}f)|)^q \, \frac{dt}{t} \\
& \lesssim \sum_{j=1}^\infty \int_{2^{-j}}^{2^{-j+1}}  (V_{\nu}(2^{-j})^{-\alpha/Q}  \, V_{\nu}^{\eta} e^{-\frac{3}{2} 2^{-j} bG_\nu}\, |(2^{-j}G_\nu)^m e^{-2^{-j-1}G_\nu}f)|)^{q} \, \frac{dt}{t} \\
& \lesssim \sum_{j=1}^\infty (e^{-6 b^{3} \, 2^{-j}G_\nu}(V_{\nu}(2^{-j})^{-\alpha/Q}  \, V_{\nu}^{\eta} \, |(2^{-j}G_\nu)^m e^{-2^{-j-1}G_\nu} f)|)^q \\
& \lesssim \sum_{j=1}^\infty   (e^{-12 b^{3} 2^{-j}G_\nu}(V_{\nu}(2^{-j})^{-\alpha/Q}  \, V_{\nu}^{\eta} \, |W_{2^{-j}}^{(m)}f)|)^q,
\end{align*}
where we used Lemma \ref{lem:heat_operator_and_derivative} (with $\kappa' = 1/2, \, \kappa = 3/2$ and $t_0 = 2^{-j}$) and Lemma \ref{lemma:ball_volume_and_heat_operator_commutation}. And, with this estimate, the inequality $\lesssim$ in \eqref{main-ineq-discrete_characterisation} follows from Proposition \ref{prop:semigroup_version_of_Fefferman_Stein_theorem}. 

\medskip 
The reverse inequality $\gtrsim$ in \eqref{main-ineq-discrete_characterisation} can be established by more or less reversing the steps of the above computations, so we omit the details. 
\end{proof}


\subsection{Characterisation in terms of vector fields} 
\label{subsec:characterisation_in_terms_of_vector_fields}
Here, we shall characterise Besov and Triebel--Lizorkin norms in terms of the vector fields. For this, let us first define
$$
W_{2^{-j}}^{(m), *} f = \sup_{t \in [2^{-j},2^{-j+1}]} \max_{|\gamma|\leq 2m} | t^m X^{\gamma} e^{-t G_\nu} f|.
$$

Unlike the previous characterisations, the following one is only for $\alpha > 0$. 
\begin{theorem}
\label{thm:vector_field_characterisation_Besov_norm}
Let $1 \leq p, \, q \leq \infty, \, \alpha > 0$ and $\eta \in \R$.
Then the classical Besov norm $\|f\|_{B^{p, q}_{\alpha, \eta}(d\mu_{\nu})}$ is equivalent to 
\begin{equation*}
\left(\sum_{j=1}^\infty \|2^{j\alpha/2} \, V_{\nu}^{\eta} \, W_{2^{-j}}^{(m), *} f\|_{L^p(d\mu_{\nu})}^q \right)^{\frac{1}{q}} + \| V_{\nu}^{\eta} \, f\|_{L^{p}(d\mu_{\nu})},    
\end{equation*}
and the non-classical Besov norm $\|f\|_{\tilde{B}^{p, q}_{\alpha, \eta}(d\mu_{\nu})}$ is equivalent to 
\begin{equation*}
\left(\sum_{j=1}^\infty \|V_{\nu}(2^{-j})^{-\alpha/Q} \, V_{\nu}^{\eta} \, W_{2^{-j}}^{(m), *} f\|_{L^p(d\mu_{\nu})}^q \right)^{\frac{1}{q}} + \|V_{\nu}^{-\alpha/Q +\eta} f\|_{L^{p}(d\mu_{\nu})},    
\end{equation*}
for any $m \in \mathbb{N}$ such that $m > \alpha/2$, with usual modifications for $q = \infty$.
\end{theorem}

\begin{theorem} \label{thm:vector_field_characterisation_Triebel_Lizorkin_norm}
Let $1 < p < \infty, \, 1 \leq  q \leq \infty, \, \alpha > 0$ and $ \eta \in \R$.
Then the classical Triebel--Lizorkin norm $\|f\|_{F^{p, \, q}_{\alpha, \eta}(d\mu_{\nu})}$ is equivalent to 
\begin{equation*}
\left\|\left(\sum_{j=1}^\infty (2^{j\alpha/2} \,  V_{\nu}^{\eta} \, W_{2^{-j}}^{(m), *} f)^q \right)^{\frac{1}{q}}\right\|_{L^p(d\mu_{\nu})} + \| V_{\nu}^{\eta} \, f\|_{L^{p}(d\mu_{\nu})},    
\end{equation*}
and the non-classical Triebel--Lizorkin norm $\|f\|_{\tilde{F}^{p, \, q}_{\alpha, \, \eta}(d\mu_{\nu})}$ is equivalent to 
\begin{equation*}
\left\|\left(\sum_{j=1}^\infty (V_{\nu}(2^{-j})^{-\alpha/Q} \, V_{\nu}^{\eta} \, W_{2^{-j}}^{(m), *} f)^q \right)^{\frac{1}{q}}\right\|_{L^p(d\mu_{\nu})} + \|V_{\nu}^{-\alpha/Q +\eta} f\|_{L^{p}(d\mu_{\nu})},    
\end{equation*}
for any $m \in \mathbb{N}$ such that $m > \alpha/2$, with usual modifications for $q = \infty$. 
\end{theorem}

We shall only prove Theorem \ref{thm:vector_field_characterisation_Triebel_Lizorkin_norm}. The proof of Theorem \ref{thm:vector_field_characterisation_Besov_norm} can be verified following similar ideas. 

\begin{proof}[Proof of Theorem \ref{thm:vector_field_characterisation_Triebel_Lizorkin_norm}] 
As earlier, we shall only prove the non-classical case. The classical case can be proved by following similar steps and with lesser difficulty. We only need to prove that 
\begin{align} 
\label{claim:vector_field_characterisation_Triebel_Lizorkin_norm}
\left\|\left(\sum_{j=1}^\infty (V_{\nu}(2^{-j})^{-\alpha/Q} \, V_{\nu}^{\eta} \, W_{2^{-j}}^{(m), *} f)^q \right)^{\frac{1}{q}}\right\|_{L^p(d\mu_{\nu})} \lesssim \|f\|_{\tilde{F}^{p,q}_{\alpha, \eta}(d\mu_{\nu})}. 
\end{align}

Recall from Lemma \ref{lem:decomposition_of_a_distibution/function} that we have the following decomposition: 
\begin{align} \label{lem:decomposition_of_a_distibution/function-again}
f = \sum_{k=0}^{m-1} \frac{1}{k!} \, W_1^{(k)} f + \frac{1}{(m-1)!} \sum_{n=1}^\infty f_n, 
\end{align}
where $f_n = \int_{2^{-n}}^{2^{-n+1}} W_s^{(m)} f \, \frac{ds}{s}$.

\medskip 
Note that for any $|\gamma| \leq 2m, \, t \in [2^{-j}, 2^{-j+1}]$ and $k \in \{0,1,\ldots,m-1\}$, we have 
\begin{align*}
|X^{\gamma} \, e^{-tG_\nu} \, W_1^{(k)} f| &= |X^{\gamma} \, e^{-tG_\nu} \, G_\nu^k \, e^{-G_\nu} f| = |X^{\gamma} \, G_\nu^k \, e^{-(t+1)G_\nu} f| \\
& \lesssim (t+1)^{-k-|\gamma|/2} \, e^{-b \, (t+1)G_\nu} |f| \, \lesssim  e^{-2b^2 \, G_\nu} |f|,
\end{align*}
where the last two steps follow from Lemma \ref{lem:heat_operator_and_derivative} (with $\kappa' = 2b, \, \kappa = 4b$ and $t_0 = 1/2$). 

\medskip 
Now, for each term in the finite sum in \eqref{lem:decomposition_of_a_distibution/function-again}, that is, for each $W_1^{(k)} f$, we have 
\begin{align}
\label{claim-part-1:vector_field_characterisation_Triebel_Lizorkin_norm} 
& \left\|\left(\sum_{j=1}^{\infty} (V_{\nu}(2^{-j})^{-\alpha/Q} V_{\nu}^{\eta} \, W_{2^{-j}}^{(m),*} (W_1^{(k)} f))^q \right)^{1/q}\right\|_{L^p(d\mu_{\nu})} \\ 
\nonumber & \lesssim \left\|\left(\sum_{j=1}^{\infty} (2^{j\alpha/2} V_{\nu}^{-\alpha/Q+ \eta} \, 2^{-jm} \, e^{- 2b^2\, G_\nu} |f|)^q \right)^{1/q}\right\|_{L^p(d\mu_{\nu})} \\ 
\nonumber & \lesssim \left\|e^{-4b^{3}G_{\nu}}\left(\sum_{j=1}^{\infty} (2^{j\alpha/2}V_{\nu}^{-\alpha/Q + \eta} 2^{-jm} |f|)^q \right)^{1/q}\right\|_{L^p(d\mu_{\nu})} \\ 
\nonumber & \lesssim \left\|\left(\sum_{j=1}^{\infty} (2^{j\alpha/2}V_{\nu}^{-\alpha/Q + \eta} 2^{-jm} |f|)^q \right)^{1/q}\right\|_{L^p(d\mu_{\nu})} \\ 
\nonumber & \lesssim \left\| \left(\sum_{j=1}^{\infty} 2^{-j q(m-\frac{\alpha}{2})} \right)^{1/q} V_{\nu}^{-\alpha/Q + \eta} \, f \right\|_{L^p(d\mu_{\nu})} \\
\nonumber & \lesssim \left\|V_{\nu}^{-\alpha/Q + \eta} f \right\|_{L^p(d\mu_{\nu})}.
\end{align}

On the other hand, for each $f_n$ appearing in \eqref{lem:decomposition_of_a_distibution/function-again}, note that 
\begin{align*}
X^{\gamma} \, e^{-tG_\nu} \, f_n &= X^{\gamma} \, e^{-(2^{-n-1}+t) G_\nu} \int_{2^{-n}}^{3 \times 2^{-n-1}} (sG_\nu)^m \, e^{-(s-2^{-n-1})G_\nu} f \, \frac{ds}{s} \\ 
& \quad + X^{\gamma} \, e^{-(2^{-n}+t)G_\nu} \int_{3 \times 2^{-n-1}}^{2^{-n+1}} (sG_\nu)^m \, e^{-(s-2^{-n})G_\nu} f \, \frac{ds}{s} \\ 
& = X^{\gamma} \, e^{-(2^{-n-1}+t)G_\nu} \int_{2^{-n-1}}^{2^{-n}} (s+2^{-n-1})^{m-1} \, G_\nu^m \, e^{-sG_\nu} f \, ds \\ 
& \quad + X^{\gamma} \, e^{-(2^{-n}+t)G_\nu} \int_{2^{-n-1}}^{2^{-n}} (s+2^{-n})^{m-1} \, G_\nu^m \, e^{-sG_\nu} f \, ds \\
& =: I_1 + I_2, 
\end{align*}
and therefore, with $|\gamma|\leq 2m$ and $t \in [2^{-j},2^{-j+1}]$, by using Lemma \ref{lem:heat_operator_and_derivative} (with $\kappa' = b, \, \kappa = 2b$ and $t_0 = 2^{-j}$), we have  
\begin{align*}
|I_1| & \lesssim (2^{-n-1}+t)^{-|\gamma|/2} \, e^{-b \, (2^{-n-1}+t)G_\nu} \int_{2^{-n-1}}^{2^{-n}} |(s+2^{-n-1})^{m-1} \, G_\nu^m \, e^{-sG_\nu} f| \, ds \\ 
& \lesssim (2^{-n-1}+t)^{-m} \, e^{-b \, (2^{-n-1}+t)G_\nu} \int_{2^{-n-1}}^{2^{-n}} |(sG_\nu)^m \, e^{-sG_\nu} f| \, \frac{ds}{s} \\
& \lesssim (2^{-n}+2^{-j})^{-m} \, e^{-b \, 2^{-n-1} G_{\nu}} \, e^{- 2b^{2}2^{-j} G_\nu} \int_{2^{-n-1}}^{2^{-n}} |W_s^{(m)} f| \, \frac{ds}{s}. 
\end{align*}
Similarly, one can show that for $I_2$, the following estimate holds true.
\begin{align*}
|I_2| \lesssim (2^{-n}+2^{-j})^{-m} \, e^{-b \, 2^{-n} G_{\nu}} \, e^{- 2b^{2}2^{-j} G_\nu} \int_{2^{-n-1}}^{2^{-n}} |W_s^{(m)} f| \, \frac{ds}{s}.    
\end{align*} 

Thus, we have 
$$
\sup_{t \in [2^{-j}, 2^{-j+1}]} \max_{|\gamma| \leq 2m} |X^{\gamma} \, e^{-tG_\nu} \, f_n| \lesssim 2^{m \min \{j,n\}} \, e^{-2b^{2} \, 2^{-j} G_\nu} F_n, 
$$
where 
\begin{align*}
F_n &= e^{-b \, 2^{-n-1}G_\nu} \int_{2^{-n-1}}^{2^{-n}} |W_s^{(m)} f| \, \frac{ds}{s} + e^{-b \, 2^{-n}G_\nu} \int_{2^{-n-1}}^{2^{-n}} |W_s^{(m)} f| \, \frac{ds}{s}. 
\end{align*}

Therefore, making use of Lemma \ref{lemma:ball_volume_and_heat_operator_commutation} and Proposition \ref{prop:semigroup_version_of_Fefferman_Stein_theorem}, we get 
\begin{align*}
& \left\|\left(\sum_{j=1}^{\infty} (V_{\nu}(2^{-j})^{-\alpha/Q}  V_{\nu}^{\eta} \, \sum_{n=1}^\infty W_{2^{-j}}^{(m),*} f_n)^q\right)^{1/q}\right\|_{L^p(d\mu_{\nu})} \\ 
& \lesssim \left\|\left(\sum_{j=1}^{\infty} (V_{\nu}(2^{-j})^{-\alpha/Q}  V_{\nu}^{\eta} \, 2^{-jm} \, e^{- 2b^{2} \, 2^{-j} G_\nu} \sum_{n=1}^{\infty} 2^{m \min\{j,n\}} F_n)^q\right)^{1/q}\right\|_{L^p(d\mu_{\nu})}  \\ 
& \lesssim \left\|\left(\sum_{j=1}^{\infty} (V_{\nu}(2^{-j})^{-\alpha/Q}  V_{\nu}^{\eta} \, 2^{-jm} \sum_{n=1}^{\infty} 2^{m \min\{j,n\}} F_n)^q\right)^{1/q}\right\|_{L^p(d\mu_{\nu})} \\ 
& \lesssim \left\|\left(\sum_{j=1}^{\infty} \left(2^{-jm} \sum_{n=1}^{\infty} 2^{m \min\{j,n\}} \left( \frac{V_{\nu}(2^{-n})}{V_{\nu}(2^{-j})} \right)^{\alpha/Q} g_n\right)^q\right)^{1/q}\right\|_{L^p(d\mu_{\nu})}, 
\end{align*}
where $g_n = V_{\nu}(2^{-n})^{-\alpha/Q} \, V_{\nu}^{\eta} \, F_n$. We shall also write $g_{n,1} = 2^{-n\alpha/2} \, g_n$ and $g_{n,2} = 2^{-nd\alpha/(2Q)} \, g_n$. Then, making use of the two types of ball volume estimates, for $1 \leq n \leq j$ and $n > j \geq 1$, the final estimate is further dominated by 
\begin{align*} 
& \left\|\left(\sum_{j=1}^{\infty}  \left(2^{-jm} \sum_{n=1}^{j} 2^{m \min\{j,n\}} \, 2^{(j-n)\alpha/2} g_n\right)^{q} \right)^{1/q} \right\|_{L^p(d\mu_{\nu})} \\
& \quad + \left\|\left(\sum_{j=1}^\infty \left(2^{-jm} \sum_{n=j+1}^{\infty} 2^{m \min\{j,n\}} \, 2^{(j-n)d\alpha/(2Q)} g_n \right)^q \right)^{1/q}\right\|_{L^p(d\mu_{\nu})}\\ 
& \leq \left\|\left(\sum_{j=1}^{\infty}  \left(2^{-j(m-\alpha/2)} \sum_{n=1}^{\infty} 2^{m \min\{j,n\}} g_{n,1}\right)^{q} \right)^{1/q} \right\|_{L^p(d\mu_{\nu})} \\
& \quad + \left\|\left( \sum_{j=1}^\infty \left(2^{-j(m-d\alpha/(2Q))} \sum_{n=1}^{\infty} 2^{m \min\{j,n\}} g_{n,2} \right)^q \right)^{1/q}\right\|_{L^p(d\mu_{\nu})} \\ 
& \lesssim \left\|\left(\sum_{n=1}^{\infty}  \left(2^{n\alpha/2} g_{n,1}\right)^{q} \right)^{1/q} \right\|_{L^p(d\mu_{\nu})} + \left\|\left(\sum_{n=1}^\infty \left(2^{nd\alpha/(2Q)} g_{n,2} \right)^q \right)^{1/q}\right\|_{L^p(d\mu_{\nu})} \\ 
& \lesssim \left\|\left(\sum_{n=1}^{\infty}  (V_{\nu}(2^{-n})^{-\alpha/Q} V_{\nu}^{\eta} \, F_n)^q  \right)^{1/q} \right\|_{L^p(d\mu_{\nu})},
\end{align*}
where we have used Lemma \ref{lem:Cor_of_Schur_lemma}.

\medskip 
One can estimate the first term in $F_n$ easily as follows: 
\begin{align*}
&e^{-b \, 2^{-n-1}G_\nu} \int_{2^{-n-1}}^{2^{-n}} |W_s^{(m)} f| \, \frac{ds}{s} \\ 
& \sim e^{-b \, 2^{-n-1}G_\nu} \int_{2^{-n-1}}^{2^{-n}} |e^{-(s-2^{-n-2)}G_\nu} \, (2^{-n-2}G_\nu)^m \, e^{-2^{-n-2}G_\nu} f| \, \frac{ds}{s} \\
& \lesssim e^{-b \, 2^{-n-1}G_\nu} \int_{2^{-n-1}}^{2^{-n}} e^{-3b \, 2^{-n-2} G_\nu}|W_{2^{-n-2}}^{(m)} f| \, \frac{ds}{s} \\ 
& \sim e^{-\frac{5}{4}b \, 2^{-n}G_\nu} |W_{2^{-n-2}}^{(m)} f|, 
\end{align*}
where the second last estimate follows from Lemma \ref{lem:heat_operator_and_derivative} (with $\kappa' = \frac{1}{4}, \, \kappa = \frac{3}{4}$ and $t_0 = 2^{-n}$). Similarly, one can show that the second term in $F_{n}$ can be dominated by $e^{-\frac{7}{4}b \, 2^{-n}G_\nu} |W_{2^{-n-2}}^{(m)} f|$.    

\medskip 
Using it in the above estimate, one can make use of Lemma \ref{lemma:ball_volume_and_heat_operator_commutation} and Proposition \ref{prop:semigroup_version_of_Fefferman_Stein_theorem} (as seen in several earlier computations) to deduce that 
\begin{align}
\label{claim-part-2:vector_field_characterisation_Triebel_Lizorkin_norm} 
& \left\|\left(\sum_{j=1}^{\infty} (V_{\nu}(2^{-j})^{-\alpha/Q}  V_{\nu}^{\eta} \, \sum_{n=1}^\infty W_{2^{-j}}^{(m),*} f_n)^q\right)^{1/q}\right\|_{L^p(d\mu_{\nu})} \\ 
\nonumber & \lesssim \left\|\left(\sum_{n=1}^{\infty}  (V_{\nu}(2^{-n})^{-\alpha/Q} V_{\nu}^{\eta} \,  |W_{2^{-n}}^{(m)} f|)^q  \right)^{1/q} \right\|_{L^p(d\mu_{\nu})}. 
\end{align}

Combining estimates \eqref{claim-part-1:vector_field_characterisation_Triebel_Lizorkin_norm} and \eqref{claim-part-2:vector_field_characterisation_Triebel_Lizorkin_norm}, we have \eqref{claim:vector_field_characterisation_Triebel_Lizorkin_norm}, and this completes the proof of Theorem \ref{thm:vector_field_characterisation_Triebel_Lizorkin_norm}. 
\end{proof}


\subsection{Recursive characterisation} \label{subsec:Recursive-characterisations} 

Following the idea of \cite[Theorem 4.5]{Bruno_Marco_Vallarino_Besov_TL_20}, one can easily prove the following recursive characterisation for classical Besov and Triebel--Lizorkin spaces, for $\alpha > 0$, and therefore we don't write its proof. 

\begin{theorem} \label{thm:recursive_characterisation_classical_spaces}
Let $\alpha > 0, \, \eta \in \R$ and $q \in [1, \infty]$.
\begin{enumerate}[(i)]
\item Let $p \in [1, \infty]$, then $f \in B^{p, \, q}_{\alpha+1, \eta}(d\mu_{\nu})$ if any only if $f, \, X_{j}f, \, X_{j, k}f \in B^{p, \, q}_{\alpha, \eta}(d\mu_{\nu})$ for every $1 \leq j \leq d_1, \, 1 \leq k \leq d_2$, and we have following norm equivalence: 
\begin{align*}
\|f\|_{B^{p, \, q}_{\alpha+1, \eta}(d\mu_{\nu})} \sim \sum_{j=1}^{d_1} \|X_{j}f\|_{B^{p, \, q}_{\alpha, \eta}(d\mu_{\nu})} + \sum_{j=1}^{d_1} \sum_{k=1}^{d_2}\|X_{j,k} f\|_{B^{p, \, q}_{\alpha, \eta}(d\mu_{\nu})} + \|f\|_{B^{p, \, q}_{\alpha, \eta}(d\mu_{\nu})}.
\end{align*}

\item Let $p \in (1, \infty)$, then $f \in F^{p, \, q}_{\alpha+1, \eta}(d\mu_{\nu})$ if any only if $f, \, X_{j}f, \, X_{j, k}f \in F^{p, \, q}_{\alpha, \eta}(d\mu_{\nu})$ for every $1 \leq j \leq d_1, \, 1 \leq k \leq d_2$, and we have following norm equivalence: 
\begin{align*}
\|f\|_{F^{p, \, q}_{\alpha+1, \eta}(d\mu_{\nu})} \sim \sum_{j=1}^{d_1} \|X_{j}f\|_{F^{p, \, q}_{\alpha, \eta}(d\mu_{\nu})} + \sum_{j=1}^{d_1} \sum_{k=1}^{d_2}\|X_{j,k} f\|_{F^{p, \, q}_{\alpha, \eta}(d\mu_{\nu})} + \|f\|_{F^{p, \, q}_{\alpha, \eta}(d\mu_{\nu})}.
\end{align*}
\end{enumerate}
\end{theorem}

We don't think that a similar characterisation holds true for non-classical spaces, but right now, we do not have any explicit argument for it. 


\subsection{Comparison of classical Besov spaces} \label{subsec:comparison_of_classical_Besov_spaces} 

In a recent work, Zhao et al. \cite{Geometric_topics_related_to_Besov_type_spaces_on_the_Grushin_setting_2005} also defined and studied Besov spaces for the Grushin operator. For $\alpha > 0, \, p \in [1,\infty), \,  q \in [1, \infty]$, they defined their Besov spaces $B^{G, \alpha, 1}_{p,\, q}(dx)$ as the space of all $f \in L^p(dx)$ such that
$$
\|f\|_{B^{G, \alpha, 1}_{p,\, q}(dx)} := \|f\|_{L^{p}(dx)} + N^{G, \alpha, 1}_{p, \, q}(f) < \infty, 
$$
where 
$$
N^{G, \alpha, 1}_{p, \, q}(f) := \left( \int_{0}^{\infty} \left(\int_{\R^{d_1+d_2}} e^{-tG}(|f-f(x)|^{p})(x) \, dx \right)^{q/p} \, \frac{dt}{t^{q \alpha/2+1}}\right)^{1/q},
$$
with appropriate changes when $q = \infty$. The authors investigated two types of Besov classes associated with the Grushin semigroup and the fractional Grushin semigroup, and established a relation between them. They also established some isoperimetric inequalities for the fractional perimeter in this setup and studied some embedding theorems. Furthermore, they proved a characterisation of the above norm by differences. Let $B(y,t)$ denote the open ball with respect to the Grushin metric $\rho$ (see the definition in \eqref{Grushin_distance}) and $|B(y,t)|$ its Lebesgue measure. They defined the Besov space $B^{G, \alpha, 2}_{p, \,q}(dx)$ as the space of all $f \in L^{p}(dx)$ such that 
$$
\|f\|_{B^{G, \alpha, 2}_{p, \, q}(dx)} := \|f\|_{L^{p}(dx)} + N^{G, \alpha, 2}_{p, \, q}(f) < \infty,
$$
where
\begin{align*} 
N^{G, \alpha, 2}_{p, \, q}(f) := \left(\int_{0}^{\infty} \left(\int_{\R^{d_1+d_2}} \int_{B(y,t)} \frac{|f(y)-f(x)|^{p}}{t^{p\alpha} \, |B(y,t)|} \, dx \, dy\right)^{q/p} \, \frac{dt}{t} \right)^{1/q},
\end{align*}
with appropriate changes when $q = \infty$. It was shown in \cite{Geometric_topics_related_to_Besov_type_spaces_on_the_Grushin_setting_2005} that both these Besov spaces $B^{G, \alpha, 1}_{p, \, q}(dx)$ and $B^{G, \alpha, 2}_{p, \, q}(dx)$ coincide with norm equivalence. 

\medskip 
We do not know whether our classical Besov spaces $B^{p, \, q}_{\alpha, \, 0}(dx)$ for the Grushin operator (without drift) coincide with the Besov spaces $B^{G, \alpha, 2}_{p, \, q}(dx)$ of \cite{Geometric_topics_related_to_Besov_type_spaces_on_the_Grushin_setting_2005} or not, but we shall now show that their Besov spaces $B^{G, \alpha, 2}_{p, \, q}(dx)$ do embed in our $B^{p, \, q}_{\alpha, \, 0}(dx)$. 

\begin{theorem} \label{thm:comparison_with_known_Besov_space}
For any $p \in [1,\infty), \, q \in [1, \infty]$ and $\alpha > 0$, we have $B^{G, \alpha, 2}_{p, \, q}(dx) \hookrightarrow B^{p, \, q}_{\alpha, 0}(dx)$. 
\end{theorem}
\begin{proof}
It is enough to prove that $\mathcal{B}^{p,q}_{\alpha, 0}(f)  \lesssim N^{G, \alpha, 2}_{p,\, q}(f)$. As usual, we shall write the details for the case when $1 \leq q < \infty$, as $q = \infty$ can be dealt with in the same manner with routine modifications. In the following computations, we shall make use of the estimate \eqref{est:heat_kernel_gradient_bounds} and the fact that the integral of the heat kernel is equal to $1$ for every $t > 0$. 

\medskip 
Let us take $m = [\alpha/2]+1$ and restrict our attention to $t \in (0,1)$. Note that 
\begin{align*}
\left| W_{t}^{(m)} f(x) \right| 
& = \left|(tG)^{m} \int_{\R^{d_1+d_2}} H_{t}(x,y) \, f(y) \, dy \right| \\ 
& = \left| \int_{\R^{d_1+d_2}} ((tG)^{m} H_{t}(x,y)) \, (f(y)-f(x)) \, dy\right| \\ 
& \lesssim \int_{\R^{d_1+d_2}} |B(y, \sqrt{t})|^{-1} \, e^{-b' \rho(x,y)^2/t} \, |f(y)-f(x)| \, dy  \\
& \lesssim \left(\int_{\R^{d_1+d_2}}|B(y, \sqrt{t})|^{-1} \, e^{-b'\rho(x,y)^2/t} \, dy\right)^{1/p'} \\
& \quad \times \left( \int_{\R^{d_1+d_2}} |B(y, \sqrt{t})|^{-1} \, e^{-b'\rho(x,y)^2/t} \, |f(y)-f(x)|^{p} \, dy \right)^{1/p} \\ 
& \lesssim \left( \int_{\R^{d_1+d_2}} |B(y, \sqrt{t})|^{-1} \, e^{-b'\rho(x,y)^2/t} \, |f(y)-f(x)|^{p} \, dy \right)^{1/p},  
\end{align*}
where in the last step we have used the well-known basic fact that 
$$ \sup_{x \in \R^{d_1+d_2}} \int_{\R^{d_1+d_2}}|B(y, \sqrt{t})|^{-1} \, e^{-b' \rho(x,y)^2/t} \, dy < \infty. $$

The above estimate implies that 
\begin{align*}
\|W_{t}^{(m)}f\|_{L^{p}(d\mu_{\nu})} 
& \lesssim \left(\int_{\R^{d_1+ d_2}} \int_{B(y, \sqrt{t})} \frac{|f(y)-f(x)|^{p}}{|B(y, \sqrt{t})|} \, dx  \, dy \right)^{1/p} \\
& \quad + \left(\int_{\R^{d_1+ d_2}} \sum_{j=0}^{\infty} \int_{2^{j}\sqrt{t} \leq \rho(x,y) < 2^{j+1}\sqrt{t}} \frac{|f(y)-f(x)|^{p} e^{-b'\rho(x,y)^2/t}}{|B(y, \sqrt{t})|} \, dx  \, dy \right)^{1/p} \\
& \lesssim \left(\int_{\R^{d_1+ d_2}} \int_{B(y, \sqrt{t})} \frac{|f(y)-f(x)|^{p}}{|B(y, \sqrt{t})|} \, dx  \, dy \right)^{1/p} \\
& \quad +  \left(\sum_{j=0}^{\infty} e^{-b' 2^{2j}} \, 2^{jQ} \int_{\R^{d_1+ d_2}} \int_{2^{j}\sqrt{t} \leq \rho(x,y) < 2^{j+1}\sqrt{t}} \frac{|f(y)-f(x)|^{p}}{|B(y, 2^{j+1}\sqrt{t})|} \, dx  \, dy \right)^{1/p} \\
& =: I_1(t) + I_2(t). 
\end{align*}

We need to further estimate $I_{2}(t)$. We will do it in two different cases, namely, $q \leq p$ and $q > p$, as follows. When $q \leq p$, that is, when $0 < q/p \leq 1$, we have 
\begin{align*}
I_{2}(t)^{q} \lesssim \sum_{j=0}^{\infty} e^{-q \, b' 2^{2j}/(2p)} \, 2^{qjQ/p} \left(\int_{\R^{d_1+ d_2}}  \int_{B(y, 2^{j+1}\sqrt{t})} \frac{|f(y)-f(x)|^{p}}{|B(y, 2^{j+1}\sqrt{t})|} \, dx  \, dy \right)^{q/p},  
\end{align*}
whereas when $q > p$, that is, when $1 < q/p < \infty$, we can apply H\"{o}lder's inequality with $r = q/p$ to get 
\begin{align*}
I_{2}(t)^{q} 
&  \lesssim \left(\sum_{j=0}^{\infty} e^{-b' 2^{2j}} 2^{jQ} \int_{\R^{d_1+ d_2}} \int_{B(y, 2^{j+1} \sqrt{t})} \frac{|f(y)-f(x)|^{p}}{|B(y, 2^{j+1} \sqrt{t})|} \, dx  \, dy \right)^{r} \\  
& \lesssim \left(\sum_{j=0}^{\infty} e^{-\frac{b'}{2}2^{2j}r'}\right)^{r/r'} \sum_{j=0}^{\infty} e^{-\frac{b'}{2} 2^{2j}r} 2^{jQr} \left(\int_{\R^{d_1+ d_2}} \int_{B(y, 2^{j+1} \sqrt{t})} \frac{|f(y)-f(x)|^{p}}{|B(y, 2^{j+1} \sqrt{t})|} \, dx  \, dy \right)^{r} \\
& \lesssim  \sum_{j=0}^{\infty} e^{-\frac{b'}{2} 2^{2j}r} 2^{jQr} \left(\int_{\R^{d_1+ d_2}} \int_{B(y, 2^{j+1} \sqrt{t})} \frac{|f(y)-f(x)|^{p}}{|B(y, 2^{j+1} \sqrt{t})|} \, dx  \, dy \right)^{r}. 
\end{align*}
Altogether, for any $1 \leq q < \infty$, we have shown that 
\begin{align*}
I_{2}(t)^{q} \lesssim \sum_{j=0}^{\infty} e^{-q \, b' 
2^{2j}/(2p)} \, 2^{qjQ/p} \left(\int_{\R^{d_1+ d_2}}  \int_{B(y, 2^{j+1}\sqrt{t})} \frac{|f(y)-f(x)|^{p}}{|B(y, 2^{j+1}\sqrt{t})|} \, dx  \, dy \right)^{q/p}.  
\end{align*}

Using the above estimate of $I_{2}(t)^{q}$, we now have 
\begin{align*}
& \mathcal{B}^{p,q}_{\alpha, 0}(f)  \lesssim \left(\int_{0}^{1} t^{-q\alpha/2} \left(I_{1}(t) + I_{2}(t) \right)^{q} \frac{dt}{t} \right)^{1/q} \\
& \lesssim \left(\int_{0}^{1} \left(\int_{\R^{d_1+ d_2}} \int_{B(y, \sqrt{t})} \frac{|f(y)-f(x)|^{p}}{t^{p\alpha/2}|B(y, \sqrt{t})|} \, dx  \, dy \right)^{q/p} \, \frac{dt}{t} \right)^{1/q} \\
& \quad + \sum_{j=0}^{\infty} e^{-b' 2^{2j}/(2p)} \, 2^{jQ/p} \left(\int_{0}^{1} \left(\int_{\R^{d_1+ d_2}} \int_{B(y, 2^{j+1}\sqrt{t})} \frac{|f(y)-f(x)|^{p}}{t^{p\alpha/2} \, |B(y, 2^{j+1}\sqrt{t})|} \, dx  \, dy \right)^{q/p} \, \frac{dt}{t} \right)^{1/q} \\
& \lesssim \left(\int_{0}^{1} \left(\int_{\R^{d_1+ d_2}} \int_{B(y, t)} \frac{|f(y)-f(x)|^{p}}{t^{p\alpha}|B(y, t)|} \, dx  \, dy \right)^{q/p} \, \frac{dt}{t} \right)^{1/q} \\
& \quad + \sum_{j=0}^{\infty} e^{-b' 2^{2j}/(2p)} \, 2^{jQ/p} \, 2^{(j+1)\alpha}\left(\int_{0}^{2^{j+1}} \left(\int_{\R^{d_1+ d_2}} \int_{B(y, t)} \frac{|f(y)-f(x)|^{p}}{t^{p\alpha} \, |B(y, t)|} \, dx  \, dy \right)^{q/p} \, \frac{dt}{t} \right)^{1/q} \\
& \lesssim \left(\int_{0}^{\infty} \left(\int_{\R^{d_1+ d_2}} \int_{B(y, t)} \frac{|f(y)-f(x)|^{p}}{t^{p\alpha}|B(y, t)|} \, dx  \, dy \right)^{q/p} \, \frac{dt}{t} \right)^{1/q} \\
& = N^{G, \alpha, 2}_{p,q}(f), 
\end{align*}
which completes the proof of the theorem.
\end{proof}


\section{Interpolation} 
\label{sec:interpolation}
Our classical and non-classical Besov and Triebel--Lizorkin spaces enjoy the following complex interpolation properties. 

\begin{theorem}
\label{thm:interpolation_Besov_spaces}
Let $p_0, \, p_1 \in [1, \infty)$ and $q_0, \, q_1 \in [1, \infty]$. Let $\alpha_0, \, \alpha_1 \geq 0$ and $\eta_0, \, \eta_1 \in \mathbb{R}$. Given any $\theta \in (0,1)$, let $\alpha_{\theta} = (1-\theta)\alpha_0 + \theta \alpha_1, \, \eta_{\theta} = (1-\theta)\eta_{0} + \theta \eta_1, \, \frac{1}{p_{\theta}} = \frac{1-\theta}{p_0} + \frac{\theta}{p_1}$, and $\frac{1}{q_{\theta}} = \frac{1-\theta}{q_0} + \frac{\theta}{q_1}$. 
\begin{enumerate}[(i)]
\item For classical Besov spaces, we have
$$
(B^{p_0, q_0}_{\alpha_0, \eta_0}(d\mu_{\nu}), \, B^{p_1, q_1}_{\alpha_1, \eta_1}(d\mu_{\nu}))_{[\theta]} = B^{p_\theta, q_\theta}_{\alpha_{\theta}, \eta_{\theta}}(d\mu_{\nu}). 
$$

\medskip 
\item For non-classical Besov spaces, we have
$$
(\tilde{B}^{p_0, q_0}_{\alpha_0, \eta_0}(d\mu_{\nu}), \, \tilde{B}^{p_1, q_1}_{\alpha_1, \eta_1}(d\mu_{\nu}))_{[\theta]} = \tilde{B}^{p_\theta, q_\theta}_{\alpha_\theta, \eta_\theta}(d\mu_{\nu}). 
$$
\end{enumerate}
\end{theorem}

\begin{theorem}
\label{thm:interpolation_Triebel_Lizorkin_spaces}
Let $p_0, \, p_1 \in (1, \infty)$ and $q_0, \, q_1 \in [1, \infty]$. Let $\alpha_0, \, \alpha_1 \geq 0$ and $\eta_0, \, \eta_1 \in \mathbb{R}$. Given any $\theta \in (0,1)$, let $\alpha_{\theta} = (1-\theta)\alpha_0 + \theta \alpha_1, \, \eta_{\theta} = (1-\theta)\eta_{0} + \theta \eta_1, \, \frac{1}{p_{\theta}} = \frac{1-\theta}{p_0} + \frac{\theta}{p_1}$, and $\frac{1}{q_{\theta}} = \frac{1-\theta}{q_0} + \frac{\theta}{q_1}$. 
\begin{enumerate}[(i)]
\item For classical Triebel--Lizorkin spaces, we have
$$
(F^{p_0, q_0}_{\alpha_0, \eta_0}(d\mu_{\nu}), \, F^{p_1, q_1}_{\alpha_1, \eta_1}(d\mu_{\nu}))_{[\theta]} = F^{p_\theta, q_\theta}_{\alpha_\theta, \eta_\theta}(d\mu_{\nu}). 
$$

\medskip 
\item For non-classical Triebel--Lizorkin spaces, we have 
$$
(\tilde{F}^{p_0, q_0}_{\alpha_0, \eta_0}(d\mu_{\nu}), \, \tilde{F}^{p_1, q_1}_{\alpha_1, \eta_1}(d\mu_{\nu}))_{[\theta]} = \tilde{F}^{p_\theta, q_\theta}_{\alpha_\theta, \eta_\theta}(d\mu_{\nu}).
$$
\end{enumerate}
\end{theorem}

As in the previous section, we shall only prove Theorem \ref{thm:interpolation_Triebel_Lizorkin_spaces}, as Theorem \ref{thm:interpolation_Besov_spaces} can be proved in a similar manner by using \cite[Theorems 5.5.3, 5.6.3] {Interpolation_spaces_book}. Unlike \cite[Theorem 6.1]{Bruno_Marco_Vallarino_Besov_TL_20}, in our interpolation Theorem \ref{thm:interpolation_Besov_spaces} for Besov spaces we cannot take $p_0$ or $p_1$ to be $\infty$. This is because we do not know how to interpolate weighted $L^\infty$-spaces while using \cite[Theorems 5.5.3, 5.6.3] {Interpolation_spaces_book}.  

\begin{proof}[Proof of Theorem \ref{thm:interpolation_Triebel_Lizorkin_spaces}] 
Once again, we shall only write the details for non-classical spaces. The idea is to show that Triebel--Lizorkin spaces $\tilde{F}^{p, q}_{\alpha, \eta}(d\mu_{\nu})$ are retracts of some well-known spaces, where interpolation is known to be true. Recall that a space $Y$ is called a retract of a space $X$ if there exist two bounded linear maps $\mathcal{P}: X \to Y$ and $\mathcal{J}: Y \to X$ such that $\mathcal{P} \circ \mathcal{J}$ is an identity map on $Y$. 

\medskip 
Given a positive measurable function $w$, consider the following space:
$$
L^{p}(l^{q}_{\alpha}, w \, d\mu_{\nu}) := \left\{ u = (u_j)_{j \geq 0} : u_j \in L^{p}(w \, d\mu_{\nu}) \, \, \text{and} \, \, \| u \|_{L^{p}(l^{q}_{\alpha}, \, w \, d\mu_{\nu})} < \infty \right\}, 
$$
where 
$$
\| u \|_{L^{p}(l^{q}_{\alpha}, \, w \, d\mu_{\nu})} := \left\|\left(\sum_{j=0}^{\infty}(2^{j\alpha/2} |u_{j}|)^{q}\right)^{1/q}\right\|_{L^{p}(w\, d\mu_{\nu})}. 
$$

In our case, we shall work with weight functions $V_{\nu}^{p\eta}$, and show that $\tilde{F}^{p, \, q}_{\alpha, \eta}(d\mu_{\nu})$ is a retract of $L^{p}(l^{q}_{\alpha}, V_{\nu}^{p\eta} \, d\mu_{\nu})$. Once we have that, our interpolation result follows from \cite[Theorem 6.4.2]{Interpolation_spaces_book} and   \cite[p. 130]{Interpolation_theory_Triebel}. 

\medskip 
Now, consider the map $\mathcal{J}$ defined on $ \tilde{F}^{p, \, q}_{\alpha, \, \eta}(d\mu_{\nu})$ by 
$$ 
\mathcal{J}f := ((\mathcal{J}f)_{j})_{j \in \mathbb{N}},$$ 
where 
$$
(\mathcal{J}f)_{j}(x) = 
\left\{ 
\begin{array}{ll}
V_{\nu}(x, 1)^{-\alpha/Q} \, e^{-\frac{1}{2}G_{\nu}} f(x), & \text{if  } j = 0, \\ 
\\ 
2^{m-j\frac{\alpha}{2}} V_{\nu}(x, 2^{-j})^{-\alpha/Q} \, W_{2^{-j-1}}^{(m)}f(x), & \text{if  } j \geq 1, 
\end{array} 
\right.
$$
and the map $\mathcal{P}$ defined on $L^{p}(l^{q}_{\alpha}, V_{\nu}^{p\eta} \, d\mu_{\nu}) $ by 
\begin{align*}
\mathcal{P}u & := \sum_{k=0}^{2m-1} \frac{1}{k!} G_{\nu}^{k} e^{-\frac{1}{2}G_{\nu}}(V_{\nu}^{\alpha/Q} u_{0}) \\
& \quad + \frac{1}{(2m-1)!}\sum_{j=1}^{\infty} 2^{jm+j\frac{\alpha}{2}} \int_{2^{-j}}^{2^{-j+1}} t^{2m} \, G_{\nu}^{m} \, e^{-(t-2^{-j-1})G_{\nu}}(V_{\nu}(2^{-j})^{\alpha/Q} \, u_{j}) \, \frac{dt}{t} \\
& =: \mathcal{P}_1u + \mathcal{P}_2u.
\end{align*}

It is quite straightforward to verify that $\mathcal{J}$ maps $\tilde{F}^{p, \, q}_{\alpha, \, \eta}(d\mu_{\nu})$ boundedly into $L^{p}(l^{q}_{\alpha}, V_{\nu}^{p\eta} \, d\mu_{\nu})$ and that $\mathcal{P}\circ \mathcal{J}$ is an identity map on $\tilde{F}^{p,\, q}_{\alpha, \eta}(d\mu_{\nu})$, so we omit the details. Hence, we shall be done if we could just show that $\mathcal{P}$ maps $L^{p}(l^{q}_{\alpha}, V_{\nu}^{p\eta} \, d\mu_{\nu})$ boundedly into $\tilde{F}^{p, \, q}_{\alpha, \, \eta}(d\mu_{\nu})$ and we prove the same here. Note that 
\begin{align*}
\|\mathcal{P}u\|_{\tilde{F}^{p, \, q}_{\alpha, \eta}(d\mu_{\nu})} 
& \lesssim \|V_{\nu}^{-\alpha/Q+\eta} \, e^{-\frac{1}{2}G_{\nu}} (\mathcal{P}_1u)\|_{L^p(d\mu_{\nu})} \\ 
& \quad + \|V_{\nu}^{-\alpha/Q+\eta} \, e^{-\frac{1}{2}G_{\nu}} (\mathcal{P}_2u)\|_{L^p(d\mu_{\nu})} \\
& \quad + \left\|\left(\sum_{j=1}^{\infty} (V_{\nu}(2^{-j})^{-\alpha/Q} \, V_{\nu}^{\eta}  |W_{2^{-j}}^{(m)}(\mathcal{P}_1u)|)^q\right)^{1/q}\right\|_{L^{p}(d\mu_{\nu})} \\
& \quad + \left\|\left(\sum_{j=1}^{\infty} (V_{\nu}(2^{-j})^{-\alpha/Q} \, V_{\nu}^{\eta}  |W_{2^{-j}}^{(m)}(\mathcal{P}_2u)|)^q\right)^{1/q}\right\|_{L^{p}(d\mu_{\nu})} \\
& =: I_1 + I_2+ I_3 + I_4, 
\end{align*}
and we will estimate each of the four $I_j$, $j = 1, \ldots, 4$, one by one. Let us begin with $I_1$. Using Lemmas \ref{lem:heat_operator_and_derivative} and \ref{lemma:ball_volume_and_heat_operator_commutation} along with the boundedness of the heat operator, we get
\begin{align*}
I_1 & \lesssim \sum_{k=0}^{2m-1} \frac{1}{k!} \|V_{\nu}^{-\alpha/Q+\eta} |G_{\nu}^{k} e^{-G_{\nu}}(V_{\nu}^{\alpha/Q} u_0)|\|_{L^{p}(d\mu_{\nu})} \\
& \lesssim \|V_{\nu}^{-\alpha/Q+\eta} e^{-bG_{\nu}}|V_{\nu}^{\alpha/Q} u_0|\|_{L^{p}(d\mu_{\nu})} \\
& \lesssim \|V_{\nu}^{\eta} \, u_0\|_{L^{p}(d\mu_{\nu})} = \|u_0\|_{L^{p}(V_{\nu}^{p\eta} \, d\mu_{\nu})} \leq \|u\|_{L^{p}(l^{q}_{\alpha}, \, V_{\nu}^{p\eta} \, d\mu_{\nu})}.
\end{align*}
Next, we consider $I_2$.
For the same, for each $j \in \mathbb{N}$ and $t \in [2^{-j}, 2^{-j+1}]$, if we write 
$$
F_{j, \,t} = t^{2m} \, G_{\nu}^{m} \, e^{-(t-2^{-j-1})G_{\nu}}(V_{\nu}(2^{-j})^{\alpha/Q} \, u_j), 
$$
then, by using Lemmas \ref{lem:heat_operator_and_derivative} and \ref{lemma:ball_volume_and_heat_operator_commutation}, we have
\begin{align*}
|F_{j, \, t}| & \lesssim t^{2m} (t-2^{-j-1})^{-m} e^{-3b^{2} 2^{-j-1}G_{\nu}} |V_{\nu}(2^{-j})^{\alpha/Q} \, u_j| \\
& \lesssim 2^{-jm} V_{\nu}(2^{-j})^{\alpha/Q}  \, e^{-3b^3 2^{-j}G_{\nu}} |u_j|, 
\end{align*}
and therefore 
\begin{align*}
I_2 & \lesssim \left\| \sum_{j=1}^{\infty} 2^{jm+j\frac{\alpha}{2}} V_{\nu}^{-\alpha/Q+\eta} \, e^{-\frac{1}{2}G_{\nu}} \int_{2^{-j}}^{2^{-j+1}} |F_{j, \, t}(\cdot)| \, \frac{dt}{t}\right\|_{L^{p}(d\mu_{\nu})} \\
& \lesssim \left\|\sum_{j=1}^{\infty} 2^{j\frac{\alpha}{2}} V_{\nu}^{-\alpha/Q+\eta} \, e^{-\frac{1}{2}G_{\nu}} \int_{2^{-j}}^{2^{-j+1}}   V_{\nu}(2^{-j})^{\alpha/Q} \, e^{-3 b^{3} 2^{-j}G_{\nu}} |u_j| \, \frac{dt}{t}\right\|_{L^{p}(d\mu_{\nu})} \\
& \lesssim \left\|\sum_{j=1}^{\infty} 2^{j\frac{\alpha}{2}} \, 2^{-j\frac{\alpha d}{2Q}} \, V_{\nu}^{\eta} \,  e^{-bG_{\nu}} \, e^{-3b^3 2^{-j}G_{\nu}} |u_j| \right\|_{L^{p}(d\mu_{\nu})} \\
& \lesssim \left\|\sum_{j=1}^{\infty} e^{-6b^4 2^{-j}G_{\nu}} (2^{j\frac{\alpha}{2}} \, 2^{-j\frac{\alpha d}{2Q}}  \, |V_{\nu}^{\eta} u_j|)\right\|_{L^{p}(d\mu_{\nu})} \\
& \lesssim \left\|\sum_{j=1}^{\infty} 2^{j\frac{\alpha}{2}} \, 2^{-j\frac{\alpha d}{2Q}} \, | \, V_{\nu}^{\eta} u_j| \right\|_{L^{p}(d\mu_{\nu})} \\
& \lesssim \left\|\left(\sum_{j=1}^{\infty} (2^{j\frac{\alpha}{2}} \, |\, V_{\nu}^{\eta} u_j|)^{q} \right)^{1/q} \right\|_{L^{p}(d\mu_{\nu})} \\
& \leq \|u\|_{L^{p}(l^{q}_{\alpha}, V_{\nu}^{p\eta} \, d\mu_{\nu})},
\end{align*}
where we used Lemma \ref{lemma:ball_volume_and_heat_operator_commutation} and Proposition \ref{prop:semigroup_version_of_Fefferman_Stein_theorem} in the intermediate steps.

\medskip Next, we estimate $I_3$. Using Lemma \ref{lem:heat_operator_and_derivative}
(with $\kappa' = b, \, \kappa = 3b$ and $t_0 = \frac{1}{2}$) and Lemma \ref{lemma:ball_volume_and_heat_operator_commutation}, we have for $t \in (0,1)$, 
\begin{align*}
|W_{t}^{(m)} \, G_{\nu}^{k} \, e^{-\frac{1}{2}G_{\nu}}(V_{\nu}^{\alpha/Q} \,  u_0)| & = |t^{m} \, G_{\nu}^{m+k} \, e^{-\left(t+\frac{1}{2}\right)G_{\nu}}(V_{\nu}^{\alpha/Q} \,  u_0)| \\
& \lesssim t^{m} V_{\nu}^{\alpha/Q} \, e^{-3b^3G_{\nu}} |u_0|, 
\end{align*}
using which along with Theorem \ref{thm:discrete_characterisation-Triebel--Lizorkin}, we get 
\begin{align*}
I_3 
& \sim \left\|\left(\int_{0}^{1} (V_{\nu}(t)^{-\alpha/Q} \, V_{\nu}^{\eta} \, |W_{t}^{(m)}(\mathcal{P}_1u)|)^q \, \frac{dt}{t} \right)^{1/q}\right\|_{L^{p}(d\mu_{\nu})} \\
& \lesssim \left\|\left(\int_{0}^{1} (t^{m-\alpha/2} \, V_{\nu}^{\eta} \, e^{-3b^3 G_{\nu}} |u_0|)^q \, \frac{dt}{t} \right)^{1/q}\right\|_{L^{p}(d\mu_{\nu})} \\
& \lesssim \|e^{-6b^4 G_{\nu}} |\, V_{\nu}^{\eta} u_0|\|_{L^{p}(d\mu_{\nu})}  \\
& \lesssim \| \, V_{\nu}^{\eta} u_0\|_{L^{p}(d\mu_{\nu})} \\
& \lesssim \|u\|_{L^{p}(l^{q}_{\alpha}, V_{\nu}^{p\eta} \, d\mu_{\nu})}. 
\end{align*}

Finally, we estimate $I_4$. For the same, note that by using Lemma \ref{lem:heat_operator_and_derivative} (with $\kappa' = b, \, \kappa = 3b$ and $t_0 = \frac{1}{2}(2^{-j}+2^{-k})$) and Lemma \ref{lemma:ball_volume_and_heat_operator_commutation}, we have
\begin{align*}
\left| W_{2^{-k}}^{(m)} \left( \mathcal{P}_2 u \right) \right| & \lesssim \left|W_{2^{-k}}^{(m)}\left(\sum_{j=1}^{\infty} 2^{jm+j\frac{\alpha}{2}} \int_{2^{-j}}^{2^{-j+1}} t^{2m} \, G_{\nu}^{m} \, e^{-(t-2^{-j-1})G_{\nu}}(V_{\nu}(2^{-j})^{\alpha/Q} \,  u_{j}) \, \frac{dt}{t}\right)\right| \\ 
& \lesssim 2^{-km} \sum_{j=1}^{\infty} 2^{jm+j\frac{\alpha}{2}} \int_{2^{-j}}^{2^{-j+1}} t^{2m} \, |G_{\nu}^{2m} \, e^{-(t-2^{-j-1}+2^{-k})G_{\nu}}(V_{\nu}(2^{-j})^{\alpha/Q} \, u_{j})| \, \frac{dt}{t} \\
& \lesssim 2^{-km} \sum_{j=1}^{\infty} 2^{jm+j\frac{\alpha}{2}} \int_{2^{-j}}^{2^{-j+1}} \frac{t^{2m}}{(2^{-j}+2^{-k})^{2m}} \, e^{-\frac{3}{2}b^2(2^{-j}+2^{-k})G_{\nu}}|V_{\nu}(2^{-j})^{\alpha/Q} \,  u_{j}| \, \frac{dt}{t} \\
& \sim 2^{-km} \sum_{j=1}^{\infty} 2^{-jm+j\frac{\alpha}{2}} \, 2^{2m \min\{j,k\}} \, e^{-\frac{3}{2}b^2(2^{-j}+2^{-k})G_{\nu}}|V_{\nu}(2^{-j})^{\alpha/Q} \, u_{j}| \\
& \lesssim 2^{-km} \sum_{j=1}^{\infty} 2^{-jm+j\frac{\alpha}{2}} \, 2^{2m \min\{j,k\}} \, e^{-\frac{3}{2}b^22^{-k}G_{\nu}} \, V_{\nu}(2^{-j})^{\alpha/Q} \,
e^{-3b^2 2^{-j}G_{\nu}}|u_{j}|,
\end{align*}
which, along with Lemma \ref{lemma:ball_volume_and_heat_operator_commutation} implies that 
\begin{align*}
& V_{\nu}(2^{-k})^{-\alpha/Q} \, V_{\nu}^{\eta}  |W_{2^{-k}}^{(m)}(\mathcal{P}_{2}u)| \\ 
& \lesssim e^{-6b^4 2^{-k} G_{\nu}} \, \left(2^{-km} \sum_{j=1}^{\infty} 2^{-jm+j\frac{\alpha}{2}} \, 2^{2m \min\{j,k\}} \, \frac{V_{\nu}(2^{-j})^{\alpha/Q}}{V_{\nu}(2^{-k})^{\alpha/Q}} \,
e^{-6b^4 2^{-j}G_{\nu}}|V_{\nu}^{\eta} \, u_{j}| \right), 
\end{align*}
and therefore it follows with the help of Proposition \ref{prop:semigroup_version_of_Fefferman_Stein_theorem} and Lemma \ref{lem:Cor_of_Schur_lemma} that 
\begin{align*}
I_4 
& \lesssim \left\|\left(\sum_{k=1}^{\infty}  \left(2^{-km} \sum_{j=1}^{\infty} 2^{-jm+j\frac{\alpha}{2}} \, 2^{2m \min\{j,k\}} \, \frac{V_{\nu}(2^{-j})^{\alpha/Q}}{V_{\nu}(2^{-k})^{\alpha/Q}} \,
e^{-6b^4 2^{-j}G_{\nu}}|V_{\nu}^{\eta} \, u_{j}|\right)^{q}\right)^{1/q}\right\|_{L^{p}(d\mu_{\nu})} \\
& \lesssim  \left\|\left(\sum_{k=1}^{\infty}  \left(2^{-km} \sum_{j=1}^{k} 2^{-jm+j\frac{\alpha}{2}} \, 2^{2m \min\{j,k\}} \, 2^{(k-j)\alpha/2} \,
e^{-6b^4 2^{-j}G_{\nu}}|V_{\nu}^{\eta} u_{j}|\right)^{q}\right)^{1/q}\right\|_{L^{p}(d\mu_{\nu})} \\
& \quad +  \left\|\left(\sum_{k=1}^{\infty}  \left(2^{-km} \sum_{j=k+1}^{\infty} 2^{-jm+j\frac{\alpha}{2}} \, 2^{2m \min\{j,k\}} \, 2^{(k-j)d\alpha/2Q} \,
e^{-6b^4 2^{-j}G_{\nu}}|V_{\nu}^{\eta} u_{j}|\right)^{q}\right)^{1/q}\right\|_{L^{p}(d\mu_{\nu})} \\
& \lesssim  \left\|\left(\sum_{k=1}^{\infty}  \left(2^{-km + k\frac{\alpha}{2}} \sum_{j=1}^{k} 2^{-jm} \, 2^{2m \min\{j,k\}}  \,
e^{-6b^4 2^{-j}G_{\nu}}|V_{\nu}^{\eta} u_{j}|\right)^{q}\right)^{1/q}\right\|_{L^{p}(d\mu_{\nu})} \\
& \quad +  \left\|\left(\sum_{k=1}^{\infty}  \left(2^{-km+k\frac{d\alpha}{2Q}} \sum_{j=k+1}^{\infty} 2^{-jm+j\frac{\alpha}{2}-j\frac{d\alpha}{2Q}} \, 2^{2m \min\{j,k\}} \, e^{-6b^4 2^{-j} G_{\nu}}|V_{\nu}^{\eta} u_{j}| \right)^{q}\right)^{1/q}\right\|_{L^{p}(d\mu_{\nu})} \\
& \lesssim  \left\|\left(\sum_{j=1}^{\infty}  \left( 2^{j\alpha/2}  \,
e^{-6b^4 2^{-j}G_{\nu}}|V_{\nu}^{\eta} u_{j}|\right)^{q}\right)^{1/q}\right\|_{L^{p}(d\mu_{\nu})} \\
& \lesssim  \left\|\left(\sum_{j=1}^{\infty}  \left( 2^{j\alpha/2} \,
|V_{\nu}^{\eta} \, u_{j}|\right)^{q}\right)^{1/q}\right\|_{L^{p}(d\mu_{\nu})} \\
& \leq \|u\|_{L^{p}(l^{q}_{\alpha}, \, V_{\nu}^{p\eta} \, d\mu_{\nu})}.
\end{align*}

Combining the above estimates of $I_1$, $I_2$, $I_3$ and $I_4$, we get that $\mathcal{P}$ maps $L^{p}(l^{q}_{\alpha}, V_{\nu}^{p\eta} \, d\mu_{\nu})$ boundedly into $\tilde{F}^{p, \, q}_{\alpha, \, \eta}(d\mu_{\nu})$, and this completes the proof of Theorem \ref{thm:interpolation_Triebel_Lizorkin_spaces}. 
\end{proof}


\section{Embedding theorems}
\label{sec:embeddings}
In this section, we shall study embedding properties of Besov and Triebel--Lizorkin spaces. We have the following theorem concerning classical Besov spaces. 

\begin{theorem} 
\label{thm:classical_besov_embeddings}
Let $p, \, p_0, \, p_1, \, q, \, q_0, \, q_1 \in [1,\infty], \, \eta \in \R$. 
\begin{enumerate}[(i)]
\item Let $\alpha_0, \, \alpha_1 \geq 0$ be such that either $\alpha_1 < \alpha_0$ or $\alpha_1=\alpha_0$ and $q_1 \geq q_0$. Then, 
\begin{align*}
B^{p,q_0}_{\alpha_0, \eta}(d\mu_{\nu}) \hookrightarrow B^{p, q_1}_{\alpha_1, \eta}(d\mu_{\nu}).    
\end{align*}

\item For any $\alpha_0 > \alpha_1 \geq 0$ such that $\alpha_0 - \alpha_1 = Q\left(\frac{1}{p_0} - \frac{1}{p_1}\right)$, we have 
\begin{align*}
B^{p_0, q}_{\alpha_0, \eta}(d\mu_{\nu}) \hookrightarrow B^{p_1, q}_{\alpha_1, \eta  + \frac{1}{p_0} -\frac{1}{p_1}}(d\mu_{\nu}).
\end{align*}
Moreover, when $\nu \neq 0, \, \eta = 0$ and  $\alpha_1 > 0$, the above embedding can not hold true for any space $B^{p_1, q}_{\alpha_1, \tau}(d\mu_{\nu})$ on the right-hand side with $0 \leq \tau \neq \frac{1}{p_0}-\frac{1}{p_1}$. 

\medskip 
\item If $\alpha > Q/p$, then 
\begin{align*}
\left\| V_{\nu}^{\eta + \frac{1}{p}} \, f \right\|_{L^{\infty}(d\mu_{\nu})} \lesssim \|f\|_{B^{p,\, q}_{\alpha, \eta}(d\mu_{\nu})}.
\end{align*} 
\end{enumerate}
\end{theorem}
\begin{proof} 
We shall not prove parts $(i)$ and $(iii)$, as these can be established similarly to the corresponding ones in the embedding of non-classical Besov spaces, which we shall state and prove shortly in Theorem \ref{thm:non_classical_besov_embeddings}. 

\medskip 
For part $(ii)$, we shall only write the details for $q < \infty$, and the case of $q = \infty$ follows with routine modifications. We are working with $p_0 < p_1$ and $\alpha_0 > \alpha_1 \geq 0$ such that $\alpha_0 - \alpha_1 = Q\left(\frac{1}{p_0} - \frac{1}{p_1}\right)$. Note that  
\begin{align*}
\|f\|_{B^{p_1, q}_{\alpha_1, \eta + \frac{1}{p_0} - \frac{1}{p_1}}(d\mu_{\nu})} & =  \left(\int_{0}^{1} \left(t^{-\alpha_1/2}\|V_{\nu}^{\eta + \frac{1}{p_0}-\frac{1}{p_1}} \, W_{t}^{(m)}f\|_{L^{p_1}(d\mu_{\nu})} \right)^{q} \frac{dt}{t} \right)^{1/q} \\
& \quad + \|V_{\nu}^{\eta + \frac{1}{p_0} - \frac{1}{p_1}} \, e^{-\frac{1}{2} G_{\nu}} f\|_{L^{p_1}(d\mu_{\nu})} \\
& =: I_1 + I_2.
\end{align*}
First, we shall work with $I_2$. Using Lemma \ref{lemma:ball_volume_and_heat_operator_commutation} and Proposition \ref{prop:L^p_L^q_type_ineq_for_heat_operator}, we have 
\begin{align*}
I_2 \lesssim \|V_{\nu}^{\frac{1}{p_0} - \frac{1}{p_1}} \, e^{-\frac{b}{2}G_{\nu}} (V_{\nu}^{\eta}|e^{-\frac{1}{4} G_{\nu}} f|)\|_{L^{p_1}(d\mu_{\nu})} \lesssim \|V_{\nu}^{\eta} e^{-\frac{1}{4} G_{\nu}} f\|_{L^{p_0}(d\mu_{\nu})} 
\lesssim \|f\|_{B^{p_0, q}_{\alpha_0, \eta}(d\mu_{\nu})},
\end{align*}
where the last step follows from Theorem \ref{thm:independence_of_parameters_Besov} with $t_0 = 1/4$.

\medskip 
Next, in order to estimate $I_1$, note that for $t \in (0, 1)$, once again using Lemma \ref{lemma:ball_volume_and_heat_operator_commutation} and Proposition \ref{prop:L^p_L^q_type_ineq_for_heat_operator}, we have 
\begin{align*}
& \|V_{\nu}^{\eta +\frac{1}{p_0} -\frac{1}{p_1}} W_{t}^{(m)} f\|_{L^{p_1}(d\mu_{\nu})} \\
& \lesssim t^{-\frac{Q}{2}\left(\frac{1}{p_0}-\frac{1}{p_1}\right)}\left\|V_{\nu}(t)^{\frac{1}{p_0}-\frac{1}{p_1}} \, e^{-btG_{\nu}} |V_{\nu}^{\eta} (tG_{\nu})^{m} e^{-\frac{t}{2}G_{\nu}}f| \right\|_{L^{p_1}(d\mu_{\nu})} \\
& \lesssim t^{-\frac{Q}{2}\left(\frac{1}{p_0}-\frac{1}{p_1}\right)} \left\|V_{\nu}^{\eta}  (tG_{\nu})^{m} e^{-\frac{t}{2}G_{\nu}}f\right\|_{L^{p_0}(d\mu_{\nu})}, 
\end{align*}
which implies that $I_1 \lesssim \|f\|_{B^{p_0, q}_{\alpha_0, \eta}(d\mu_{\nu})}$, and this completes the proof of the claimed embedding. We are left with showing that the embedding $B^{p_0, q}_{\alpha_0, 0}(d\mu_{\nu}) \hookrightarrow B^{p_1, q}_{\alpha_1, \tau}(d\mu_{\nu})$ is not possible for any $0 \leq \tau \neq \frac{1}{p_0}-\frac{1}{p_1}$. We postpone it for now and will prove it later in Lemma \ref{lem:sharpness_classical_theorem} 
\end{proof}

The following theorem is an analogue of the previous result for non-classical Besov spaces. One should note the change in the role of the weight functions $V_\nu^\eta$ in each of the three parts of these two theorems. 

\begin{theorem} 
\label{thm:non_classical_besov_embeddings}
Let $p, \, p_0, \, p_1, \, q, \, q_0, \, q_1 \in [1,\infty]$ and $\eta \in \mathbb{R}$.
\begin{enumerate}[(i)]
\item Let $\alpha_0, \, \alpha_1 \geq 0$ be such that either $\alpha_1 < \alpha_0$ or $\alpha_1=\alpha_0$ and $q_1 \geq q_0$. Then, 
\begin{align*}
\tilde{B}^{p,q_0}_{\alpha_0, \eta}(d\mu_{\nu}) \hookrightarrow \tilde{B}^{p,q_1}_{\alpha_1, \eta + (\alpha_1-\alpha_0)/Q}(d\mu_{\nu}).    
\end{align*}

\item For any $\alpha_0 > \alpha_1 \geq 0$ such that $\alpha_0 - \alpha_1 = Q\left(\frac{1}{p_0} - \frac{1}{p_1}\right)$, we have 
\begin{align*}
\tilde{B}^{p_0, q}_{\alpha_0, \eta}(d\mu_{\nu}) \hookrightarrow \tilde{B}^{p_1, q}_{\alpha_1, \eta}(d\mu_{\nu}).
\end{align*}

\item If $\alpha > Q/p$, then 
\begin{align*}
\left\| V_{\nu}^{\eta+\frac{1}{p}}  \, f \right\|_{L^{\infty}(d\mu_{\nu})} \lesssim \|f\|_{\tilde{B}^{p,\, q}_{\alpha, \eta+ \alpha/Q}(d\mu_{\nu})}.
\end{align*}
\end{enumerate}
\end{theorem}
\begin{proof}
We shall only write the details for $q, \, q_0, \, q_1 < \infty$, and the case when any of $q, \, q_0$ or $q_1$ is $\infty$ follows with routine modifications. 

\medskip 
\textbf{\underline{Part $(i)$}:} 
Let us first consider the case when $\alpha_1 < \alpha_0$ and $q_0 > q_1$. For $0<t<1$, making use of the ball volume estimates $V_{\nu}(x, 1)^{\frac{\alpha_1 -\alpha_0}{Q}} \lesssim t^{d(\alpha_0-\alpha_1)/2Q} V_{\nu}(x, t)^{\frac{\alpha_1 -\alpha_0}{Q}},$ we get 
\begin{align*}
& \left(\int_{0}^{1} \|V_{\nu}(t)^{-\alpha_1/Q} V_{\nu}^{\eta+(\alpha_1-\alpha_0)/Q} \, |W_t^{(m)} f|\|_{L^p(d\mu_{\nu})}^{q_1} \, \frac{dt}{t} \right)^{1/q_1} \\ 
& \lesssim \left(\int_{0}^{1} t^{\frac{dq_1(\alpha_0-\alpha_1)}{2Q}} \left\|V_{\nu}(t)^{-\alpha_0/Q}V_{\nu}^{\eta} \, |W_t^{(m)} f|\right\|_{L^p(d\mu_{\nu})}^{q_1} \, \frac{dt}{t} \right)^{1/q_1} \\ 
& \lesssim \left(\int_{0}^{1}\left\|V_{\nu}(t)^{-\alpha_0/Q}V_{\nu}^{\eta} \, |W_t^{(m)} f|\right\|_{L^p(d\mu_{\nu})}^{q_0} \, \frac{dt}{t} \right)^{1/q_0},
\end{align*}
where the last step follows from H\"{o}lder's inequality with $s= q_0/q_1$. 

\medskip 
In the remaining case, that is, when $\alpha_1 \leq \alpha_0$ and $q_1 \geq q_0$, we make use of the characterisation from Theorem \ref{thm:discrete-characterisation-Besov} and estimate 
\begin{align*}
& \left(\sum_{j=1}^{\infty} \|V_{\nu}(2^{-j})^{-\alpha_1/Q} V_{\nu}^{\eta+(\alpha_1-\alpha_0)/Q} \, |W_{2^{-j}}^{(m)} f|\|_{L^p(d\mu_{\nu})}^{q_1}\right)^{1/q_1} \\ 
& \lesssim \left(\sum_{j=1}^{\infty} \|V_{\nu}(2^{-j})^{-\alpha_0/Q} V_{\nu}^{\eta} \, |W_{2^{-j}}^{(m)} f|\|_{L^p(d\mu_{\nu})} ^{q_1}\right)^{1/q_1} \\ 
& \lesssim \left(\sum_{j=1}^{\infty} \|V_{\nu}(2^{-j})^{-\alpha_0/Q} V_{\nu}^{\eta} \, |W_{2^{-j}}^{(m)} f|\|_{L^p(d\mu_{\nu})}^{q_0}\right)^{1/q_0},    
\end{align*}
where the first step follows from $V_{\nu}(x,1)^{\frac{\alpha_1-\alpha_0}{Q}} \lesssim V_{\nu}(x, 2^{-j})^{\frac{\alpha_1-\alpha_0}{Q}},$ whereas the last step is an application of the embedding of $l^{q_0}$ into $l^{q_1}$. This completes the proof of part $(i)$. 

\medskip 
\textbf{\underline{Part $(ii)$}:} 
We are working with $p_0 < p_1$ and $\alpha_0 > \alpha_1 \geq 0$ such that $\alpha_0 - \alpha_1 = Q\left(\frac{1}{p_0} - \frac{1}{p_1}\right)$. Here also, we shall work with the definition of the Besov norm coming from Theorem \ref{thm:discrete-characterisation-Besov}, that is,   
\begin{align*}
\|f\|_{\tilde{B}^{p_1, q}_{\alpha_1, \eta}(d\mu_{\nu})} & \sim  \left(\sum_{j=1}^{\infty} \|V_{\nu}(2^{-j})^{-\alpha_1/Q} \, V_{\nu}^{\eta} \, |W_{2^{-j}}^{(m)}f|\|_{L^{p_1}(d\mu_{\nu})}^{q} \right)^{1/q} \\
& \quad + \|V_{\nu}^{-\alpha_1/Q +\eta} \, e^{-\frac{1}{2} G_{\nu}} f\|_{L^{p_1}(d\mu_{\nu})} \\
& =: I_1 + I_2,
\end{align*}
where we take $m$ to be an integer such that $m > \alpha_0/2$.

\medskip
First, we shall work with $I_2$. Using Lemma \ref{lemma:ball_volume_and_heat_operator_commutation} and Proposition \ref{prop:L^p_L^q_type_ineq_for_heat_operator}, we have 
\begin{align*}
I_2 
& \lesssim \|V_{\nu}^{\frac{1}{p_0} - \frac{1}{p_1}} \, e^{-\frac{b}{2}G_{\nu}} |V_{\nu}^{-\alpha_0/Q +\eta} \, e^{-\frac{1}{4} G_{\nu}} f|\|_{L^{p_1}(d\mu_{\nu})} \\
& \lesssim \|V_{\nu}^{-\alpha_0/Q+\eta} \, e^{-\frac{1}{4} G_{\nu}} f\|_{L^{p_0}(d\mu_{\nu})} \\
& \lesssim \|f\|_{\tilde{B}^{p_0, q}_{\alpha_0, \eta}(d\mu_{\nu})},
\end{align*}
where the last step follows from Theorem \ref{thm:independence_of_parameters_Besov} with $t_0 = 1/4$. 

\medskip 
Next, in order to estimate $I_1$, note that once again using Lemma \ref{lemma:ball_volume_and_heat_operator_commutation} and Proposition \ref{prop:L^p_L^q_type_ineq_for_heat_operator}, we have 
\begin{align*}
& \|V_{\nu}(2^{-j})^{-\alpha_1/Q} \, V_{\nu}^{\eta} \, |W_{2^{-j}}^{(m)}f|\|_{L^{p_1}(d\mu_{\nu})} \\
& \lesssim \left\| V_{\nu}(2^{-j})^{\frac{1}{p_0}-\frac{1}{p_1}} \, e^{-2^{-j+1} b^2 G_{\nu}} \, |V_{\nu}(2^{-j})^{-\alpha_0/Q}\, V_{\nu}^{\eta} \, W_{2^{-j-1} }^{(m)}f| \right\|_{L^{p_1}(d\mu_{\nu})} \\
& \lesssim \left\| V_{\nu}(2^{-j})^{-\alpha_0/Q}\, V_{\nu}^{\eta} \, W_{2^{-j-1}}^{(m)}f \right\|_{L^{p_0}(d\mu_{\nu})}, 
\end{align*}
which implies that $I_1 \lesssim \|f\|_{\tilde{B}^{p_0, q}_{\alpha_0, \eta}(d\mu_{\nu})}$, and this completes the proof of part $(ii)$. 

\medskip 
\textbf{\underline{Part $(iii)$}:} 
We are given that $\alpha > Q/p \geq 0$. Let $\alpha_1 > 0$ be such that $\alpha-\alpha_1 = Q/p$. In view of part $(ii)$ as above as well as Theorem \ref{thm:independence_of_parameters_Besov} with $t_0 = 0$, we have 
\begin{align*}
\|V_{\nu}^{\eta+\frac{1}{p}} \, f\|_{L^{\infty}(d\mu_{\nu})} = \|V_{\nu}^{- \frac{\alpha_1}{Q} + \eta+\frac{\alpha}{Q}} \, f\|_{L^{\infty}(d\mu_{\nu})} \leq \|f\|_{\tilde{B}^{\infty, q}_{\alpha_1, \eta + \alpha/Q}(d\mu_{\nu})} \lesssim \|f\|_{\tilde{B}^{p,q}_{\alpha, \eta+ \alpha/Q}(d\mu_{\nu})}, 
\end{align*}
which is exactly the claim of part $(iii)$, and this completes the proof of the theorem.
\end{proof}


The following theorem compares Besov and Triebel--Lizorkin spaces. It is an analogue of \cite[Theorem 5.3]{Bruno_Marco_Vallarino_Besov_TL_20} with no major changes in the proof, so we omit its proof. 

\begin{theorem} \label{thm:relation_between_non_classical_besov_and_triebel_lizorkin_spaces}
Let $p \in (1, \infty), \, q \in [1, \infty], \, \alpha \geq 0$ and $\eta \in \mathbb{R}$, then
$$
B^{p, \, \min\{p,q\}}_{\alpha, \eta}(d\mu_{\nu}) \hookrightarrow F^{p, \, q}_{\alpha, \eta}(d\mu_{\nu}) \hookrightarrow B^{p, \, \max\{p,q\}}_{\alpha, \eta}(d\mu_{\nu}),
$$
and 
$$
\tilde{B}^{p, \, \min\{p,q\}}_{\alpha, \, \eta}(d\mu_{\nu}) \hookrightarrow \tilde{F}^{p, \, q}_{\alpha, \, \eta}(d\mu_{\nu}) \hookrightarrow \tilde{B}^{p, \, \max\{p,q\}}_{\alpha, \, \eta}(d\mu_{\nu}).
$$
\end{theorem}


In the following two theorems we shall show analogous embedding properties of classical and non-classical Triebel--Lizorkin spaces. Before stating these theorems, let us recall that in \cite{GG-preprint-Sobolev-Grushin-2026}, we had defined and studied Sobolev spaces $L^{p}_{\alpha}(d\mu_{\nu})$ associated with $G_{\nu}$. For a given $p \in (1, \infty)$ and $\alpha \geq 0$, these spaces are defined as 
\begin{align} 
\label{eq:def_of_norms_sobolev_space}
L^{p}_{\alpha}(d\mu_{\nu}) := \{f \in L^p(d\mu_\nu): G_{\nu}^{\alpha/2}f \in L^p(d\mu_\nu)\},    
\end{align}
endowed with the norm 
\begin{align*}
\|f\|_{L^p_\alpha(d\mu_{\nu})} := \|f\|_{L^p(d\mu_\nu)}+\|G_{\nu}^{\alpha/2}f\|_{L^p(d\mu_\nu)}.
\end{align*} 
In part $(iii)$ of the next two theorems, we shall discuss the relation between these Sobolev spaces and the Triebel--Lizorkin spaces $F^{p,2}_{\alpha, 0}(d\mu_{\nu})$ and $\tilde{F}^{p,2}_{\alpha,\eta}(d\mu_{\nu})$ 

\begin{theorem} \label{thm:Triebel_Lizorkin_classical_embeddings}
Let $p, \, p_0, \, p_1, \, q, \, q_0, \, q_1, \, r \in [1,\infty]$ and $\eta \in \R$.
\begin{enumerate}[(i)]
\item Let $\alpha_0, \, \alpha_1 \geq 0$ be such that either $\alpha_1 < \alpha_0$ or $\alpha_1=\alpha_0$ and $q_1 \geq q_0$. Then 
\begin{align*}
F^{p,q_0}_{\alpha_0, \eta}(d\mu_{\nu}) \hookrightarrow F^{p,q_1}_{\alpha_1, \eta}(d\mu_{\nu}).
\end{align*} 

\item Let $1 < p_0 < p_1 < \infty$ and $\alpha_0 > \alpha_1 \geq 0$ be such that $\alpha_0 - \alpha_1 = Q \left(\frac{1}{p_0} - \frac{1}{p_1}\right)$. Also, if we take $p_0 \leq r$ and $q \leq \min\{p_1, \, r\}$, we have 
$$
F^{p_0, q}_{\alpha_0, \, \eta}(d\mu_{\nu}) \hookrightarrow F^{p_1, r}_{\alpha_1, \, \eta + \frac{1}{p_0} - \frac{1}{p_1}} (d\mu_{\nu}).
$$
Moreover, when $\nu \neq 0, \, \eta = 0$ and  $\alpha_1 > 0$, the above embedding can not hold true for any space $F^{p_1, r}_{\alpha_1, \tau}(d\mu_{\nu})$ on the right-hand side with $0 \leq \tau \neq \frac{1}{p_0}-\frac{1}{p_1}$.  

\item Let $p \in (1, \infty)$ and $\alpha \geq 0$. Then $$
L^{p}_{\alpha}(d\mu_{\nu}) = F^{p, \, 2}_{\alpha, 0}(d\mu_{\nu}),
$$
where $L^{p}_{\alpha}(d\mu_{\nu})$ stands for the Sobolev space given in \eqref{eq:def_of_norms_sobolev_space}. 

\medskip 
\item Let $p \in (1, \infty)$ and $\alpha > Q/p$, then we have
$$ \| V_{\nu}^{1/p}  f\|_{L^\infty(d\mu_{\nu})} \lesssim \|f\|_{F^{p,q}_{\alpha, 0} (d\mu_{\nu})}. $$
\end{enumerate}
\end{theorem}

\begin{theorem} \label{thm:Triebel_Lizorkin_non_classical_embeddings}
 Let $p, \, p_0, \, p_1, \, q, \, q_0, \, q_1, \, r \in [1,\infty]$ and $\eta \in \mathbb{R}$. 
\begin{enumerate}[(i)]
\item Let $\alpha_0, \, \alpha_1 \geq 0$ be such that either $\alpha_1 < \alpha_0$ or $\alpha_1=\alpha_0$ and $q_1 \geq q_0$. Then 
\begin{align*}
\tilde{F}^{p,q_0}_{\alpha_0, \eta}(d\mu_{\nu}) \hookrightarrow \tilde{F}^{p,q_1}_{\alpha_1, \eta + (\alpha_1-\alpha_0)/Q}(d\mu_{\nu}).
\end{align*} 

\item Let $1 < p_0 < p_1 < \infty$ and $\alpha_0 > \alpha_1 \geq 0$ be such that $\alpha_0 - \alpha_1 = Q \left(\frac{1}{p_0} - \frac{1}{p_1}\right)$. Also, if we take $p_0 \leq r$ and $q \leq \min\{p_1, \, r\}$, we have 
$$
\tilde{F}^{p_0, q}_{\alpha_0, \, \eta}(d\mu_{\nu}) \hookrightarrow \tilde{F}^{p_1, r}_{\alpha_1, \, \eta} (d\mu_{\nu}).
$$

\item Let $p \in (1, \infty)$ and $\alpha \geq 0$. Then $$
L^{p}_{\alpha}(d\mu_{\nu}) \hookrightarrow \tilde{F}^{p, \, 2}_{\alpha, \, \alpha/Q}(d\mu_{\nu}) \hookrightarrow L^{p}_{\alpha d/Q}(d\mu_{\nu}),
$$
where $L^{p}_{\alpha}(d\mu_{\nu})$ stands for the Sobolev space given in \eqref{eq:def_of_norms_sobolev_space}.
\end{enumerate}
\end{theorem}

We have some remarks on the above results. 
\begin{remark}
\label{rem:ABCD}
The following remarks on Theorems \ref{thm:Triebel_Lizorkin_classical_embeddings} and \ref{thm:Triebel_Lizorkin_non_classical_embeddings} are in order. 
\begin{enumerate}[(i)]
\item 
One may note the difference in parts $(iii)$ of both of these theorems. Unlike the classical case of $F^{p, \, 2}_{\alpha, 0}(d\mu_{\nu})$, we don't get a direct equivalence of the Triebel--Lizorkin space $\tilde{F}^{p, \, 2}_{\alpha, \, \alpha/Q}(d\mu_{\nu})$ with a single Sobolev space. Instead, we have a two-sided embedding with adjustment in the order of smoothness. 

\medskip 
\item 
In both of these theorems, our result in part $(ii)$ works only for limited indices, namely, when $p_0 \leq r$ and $q \leq \min\{p_1, \, r\}$. The direct arguments for proving the embedding do not work in our case. Instead, we make use of Theorems 
\ref{thm:non_classical_besov_embeddings} and 
\ref{thm:relation_between_non_classical_besov_and_triebel_lizorkin_spaces} to study these embeddings, and in doing so we get the mentioned constraints on the choice of parameters. We believe that this result should hold true without these constraints, but we do not have a proof right now. However, if we consider $\nu = 0$, then following the proof of \cite[Theorem 5.2 $(ii)$]{Bruno_Marco_Vallarino_Besov_TL_20}, one can easily prove that the embedding for the classical Triebel--Lizorkin spaces $F^{p_0, q}_{\alpha_0, 0}(dx) \hookrightarrow F^{p_1, r}_{\alpha_1, 0}(dx)$ holds true for any $q, \, r \in [1, \infty]$. 

\medskip \item 
We are unable to establish an analogue of part $(iv)$ of Theorem \ref{thm:Triebel_Lizorkin_classical_embeddings} in Theorem \ref{thm:Triebel_Lizorkin_non_classical_embeddings}. We do expect that the following embedding should hold true, but we don't know how to prove it (if true). Given $p \in (1, \infty)$ and $\alpha > Q/p$, we expect that 
$$\left \|V_{\nu}^{1/p} \, f \right\|_{L^\infty(d\mu_{\nu})} \lesssim \|f\|_{\tilde{F}^{p,q}_{\alpha, \, \alpha/Q} (d\mu_{\nu})}.$$
\end{enumerate} 
\end{remark}

Let us now sketch a proof of Theorem \ref{thm:Triebel_Lizorkin_classical_embeddings}. We shall only prove part $(iv)$ of this theorem as the proof of parts $(i)$ and $(ii)$ can be done on the same lines as that of the analogous ones in Theorem \ref{thm:Triebel_Lizorkin_non_classical_embeddings}, whereas the proof of part $(iii)$ can be written simply following the proof of \cite[Theorem 5.2 $(iii)$]{Bruno_Marco_Vallarino_Besov_TL_20}. The claim regarding the non-embedding with any $0 \leq \tau \neq \frac{1}{p_0} - \frac{1}{p_1}$ in part $(ii)$ will be taken up shortly in Lemma \ref{lem:sharpness_classical_theorem}. 

\medskip 
Finally, for part $(iv)$, given $\alpha > Q/p$, let us choose $\beta$ such that $\alpha > \beta > Q/p$, and then it follows from parts $(i)$ and $(iii)$ of this theorem and \cite[Theorem 1.3 (ii)]{GG-preprint-Sobolev-Grushin-2026} that 
$$ \| V_{\nu}^{1/p}  f\|_{L^\infty(d\mu_{\nu})} \lesssim \|f\|_{L^{p}_{\beta}(d\mu_{\nu})} \sim \|f\|_{F^{p,2}_{\beta, 0} (d\mu_{\nu})} \lesssim \|f\|_{F^{p,q}_{\alpha, 0} (d\mu_{\nu})},$$ 
completing the claim of part $(iv)$.

\begin{proof}[Proof of Theorem \ref{thm:Triebel_Lizorkin_non_classical_embeddings}]
Part $(i)$ can be proved following the proof of part $(i)$ of Theorem \ref{thm:non_classical_besov_embeddings}, so we omit the details. For part $(ii)$, let us choose an $s$ such that $q \leq s \leq r$ and $p_0 \leq s \leq p_1$. Then, using Theorems \ref{thm:non_classical_besov_embeddings} and \ref{thm:relation_between_non_classical_besov_and_triebel_lizorkin_spaces}, we have
\begin{align*}
\tilde{F}^{p_0, \, q}_{\alpha_0, \, \eta}(d\mu_{\nu}) \hookrightarrow \tilde{F}^{p_0, \, s}_{\alpha_0, \, \eta}(d\mu_{\nu}) \hookrightarrow \tilde{B}^{p_0, \, s}_{\alpha_0, \, \eta}(d\mu_{\nu}) \hookrightarrow \tilde{B}^{p_1, \, s}_{\alpha_1, \, \eta}(d\mu_{\nu}) \hookrightarrow \tilde{F}^{p_1, \, s}_{\alpha_1, \, \eta}(d\mu_{\nu})  \hookrightarrow \tilde{F}^{p_1, \, r}_{\alpha_1, \, \eta}(d\mu_{\nu}),
\end{align*}
which proves the claimed embedding. 

\medskip 
Finally, in view of the estimates for $t \in (0,1)$, 
$$
t^{-d\alpha/2Q} \, V_{\nu}(x, 1)^{-\alpha/Q} \lesssim V_{\nu}(x, t)^{-\alpha/Q} \lesssim t^{-\alpha/2} \, V_{\nu}(x, 1)^{-\alpha/Q}, 
$$ 
part $(iii)$ follows from part $(iii)$ of Theorem \ref{thm:Triebel_Lizorkin_classical_embeddings}. 
\end{proof}


We now address the role of $\tau = \frac{1}{p_0} - \frac{1}{p_1}$ mentioned in parts $(ii)$ of Theorems \ref{thm:classical_besov_embeddings} and \ref{thm:Triebel_Lizorkin_classical_embeddings} in the case when $\eta = 0$. 
\begin{lemma} \label{lem:sharpness_classical_theorem}
Let $p_0, \,  p_1 \in (1, \infty), \, q, \, r \in [1, \infty], \, \nu \neq 0$ and $\alpha_0, \, \alpha_1 > 0$. If the following embedding holds true 
\begin{align*}
F^{p_0, q}_{\alpha_0, 0}(d\mu_{\nu}) \hookrightarrow F^{p_1, r}_{\alpha_1, \tau}(d\mu_{\nu}), 
\end{align*}
for some $\tau \geq 0$, then we must have $\tau = \frac{1}{p_0} - \frac{1}{p_1}$. The same result holds true for the classical Besov spaces as well. 
\end{lemma}
\begin{proof}
Working first with the classical Triebel--Lizorkin spaces, assume that the mentioned embedding holds true for some $\tau \geq 0$. Now, choose $\tilde{\alpha}_0$ such that $\tilde{\alpha}_0 > \alpha_0$. Using parts $(i)$ and $(iii)$ of  Theorem \ref{thm:Triebel_Lizorkin_classical_embeddings}, we have
\begin{align*}
\|V_{\nu}^{\tau} \, f\|_{L^{p_1}(d\mu_{\nu})} \lesssim \|f\|_{F^{p_1, r}_{\alpha_1, \tau}(d\mu_{\nu})} \lesssim \|f\|_{F^{p_0, q}_{\alpha_0, 0}(d\mu_{\nu})} \lesssim \|f\|_{F^{p_0, 2}_{\tilde{\alpha}_0, 0}(d\mu_{\nu})} \sim \|f\|_{L^{p_0}_{\tilde{\alpha}_0}(d\mu_{\nu})},
\end{align*}
which in particular implies that 
$$ \|V_{\nu}^{\tau} \, f\|_{L^{p_1}(d\mu_{\nu})} \lesssim \|f\|_{L^{p_0}_{\tilde{\alpha}_0}(d\mu_{\nu})}, $$
but thanks to \cite[Theorem 1.3 $(iii)$]{GG-preprint-Sobolev-Grushin-2026}, the above inequality is not possible unless $\tau = \frac{1}{p_0} - \frac{1}{p_1}$. 

\medskip 
We shall use the above result on the classical Triebel--Lizorkin spaces to prove the same for the classical Besov spaces. So, let us assume that $B^{p_0, q}_{\alpha_0, 0}(d\mu_{\nu}) \hookrightarrow B^{p_1, r}_{\alpha_1, \tau}(d\mu_{\nu}),$ for some $\tau \geq 0$. Choose and fix $0 < \epsilon < \alpha_1$. Using Theorems \ref{thm:classical_besov_embeddings} and \ref{thm:relation_between_non_classical_besov_and_triebel_lizorkin_spaces}, we get
\begin{align*}
F^{p_0, p_0}_{\alpha_0 + \epsilon, 0}(d\mu_{\nu}) \hookrightarrow B^{p_0, p_0}_{\alpha_0 + \epsilon, 0}(d\mu_{\nu}) \hookrightarrow B^{p_0, q}_{\alpha_0, 0}(d\mu_{\nu}) \hookrightarrow B^{p_1, r}_{\alpha_1, \tau}(d\mu_{\nu}) \hookrightarrow B^{p_1, p_1}_{\alpha_1-\epsilon, \tau}(d\mu_{\nu}) \hookrightarrow F^{p_1, p_1}_{\alpha_1-\epsilon, \tau}(d\mu_{\nu}), 
\end{align*}
which is already shown to be not possible unless $\tau = \frac{1}{p_0} - \frac{1}{p_1}$.
\end{proof}


\section{Fractional Leibniz Rules}
\label{sec:algebra_properties}
In this section, we study the fractional Leibniz rules for the Besov and Triebel--Lizorkin spaces. 

\begin{theorem} \label{thm:algebra_properties_Besov_spaces}
Let $p, \, p_1, \,  p_2, \, p_3, \, p_4 \in [1,\infty)$ and $q \in [1, \infty]$ be such that $\frac{1}{p} = \frac{1}{p_1} + \frac{1}{p_2} = \frac{1}{p_3} + \frac{1}{p_4}$. Given any $\alpha > 0$ and $\eta \in \R$, we have 
\begin{equation*}
\|fg\|_{\tilde{B}^{p, q}_{\alpha, \eta}(d\mu_{\nu})} \lesssim \|f\|_{\tilde{B}^{p_1,q}_{\alpha, \eta}(d\mu_{\nu})} \|g\|_{L^{p_2}(d\mu_{\nu})} + \|f\|_{L^{p_3}(d\mu_{\nu})} \|g\|_{\tilde{B}^{p_4,q}_{\alpha, \eta}(d\mu_{\nu})},    
\end{equation*}
for all $f \in \tilde{B}^{p_1,q}_{\alpha, \eta}(d\mu_{\nu}) \cap L^{p_3}(d\mu_{\nu})$ and $g \in \tilde{B}^{p_4, q}_{\alpha, \eta}(d\mu_{\nu}) \cap L^{p_2}(d\mu_{\nu})$. 
\end{theorem}

\begin{theorem} \label{thm:algebra_properties_Triebel_Lizorkin_spaces}
Let $p, \, p_1, \, p_2, \, p_3, \, p_4 \in (1,\infty)$ and $q \in [1, \infty]$ be such that $\frac{1}{p} = \frac{1}{p_1} + \frac{1}{p_2} = \frac{1}{p_3} + \frac{1}{p_4}$. Given any $\alpha > 0$ and $\eta \in \R$, we have 
\begin{equation*} 
\|fg\|_{\tilde{F}^{p, q}_{\alpha, \eta}(d\mu_{\nu})} \lesssim \|f\|_{\tilde{F}^{p_1,q}_{\alpha, \eta}(d\mu_{\nu})} \|g\|_{L^{p_2}(d\mu_{\nu})} + \|f\|_{L^{p_3}(d\mu_{\nu})} \|g\|_{\tilde{F}^{p_4,q}_{\alpha,  \eta}(d\mu_{\nu})},   
\end{equation*}
for all $f \in \tilde{F}^{p_1,q}_{\alpha, \, \eta}(d\mu_{\nu}) \cap L^{p_3}(d\mu_{\nu})$ and $g \in \tilde{F}^{p_4, q}_{\alpha, \, \eta}(d\mu_{\nu}) \cap L^{p_2}(d\mu_{\nu})$. 
\end{theorem}

Analogous results hold true for classical function spaces as well, with exactly the same conditions. Here also, we shall only prove Theorem \ref{thm:algebra_properties_Triebel_Lizorkin_spaces} as Theorem \ref{thm:algebra_properties_Besov_spaces} can be proved in a similar manner.
\begin{proof}[Proof of Theorem \ref{thm:algebra_properties_Triebel_Lizorkin_spaces}]
We use the standard paraproduct technique for the proof of the theorem. Following the proof of \cite[Proposition 5.2]{Feneuil_Joseph_Algebra_properties_2018}, one can easily verify that for any $m \in \mathbb{N}$, we have the following decomposition:  
\begin{align*} 
fg = \Pi_f(g) + \Pi_g(f) + \Pi (f,g) + \sum_{h,k,n=0}^{m-1} \frac{1}{h! \, k! \, n!} W_1^{(h)} (W_1^{(k)} f \, W_1^{(n)} g),
\end{align*}
where 
\begin{align*}
\Pi_f(g) & = \sum_{h,k=0}^{m-1} \frac{1}{(m-1)! \, h! \, k!} \int_{0}^{1} W_t^{(h)} (W_t^{(m)} f \, W_t^{(k)} g) \, \frac{dt}{t}, \\
 \text{and} \qquad \Pi(f,g) & = \sum_{h,k =0}^{m-1} \frac{1}{(m-1)! \, h! \, k!} \int_{0}^{1} W_t^{(m)} (W_t^{(h)} f \, W_t^{(k)} g) \, \frac{dt}{t}.
\end{align*}

In the rest of the proof, let us fix $m = [\alpha/2]+1$. In view of the above decomposition, it is clear that the theorem would follow once we prove the following four estimates: 
\begin{align} 
\|V_{\nu}^{-\alpha/Q+\eta} f g\|_{L^p(d\mu_{\nu})} & \leq \|f\|_{\tilde{F}^{p_1, q}_{\alpha, \eta}(d\mu_{\nu})} \|g\|_{L^{p_2}(d\mu_{\nu})}, 
\label{eq:claim_1} \\ 
\mathcal{\tilde{F}}^{p,q}_{\alpha, \eta}(W_1^{(h)} (W_1^{(k)} f \, W_1^{(n)} g)) & \lesssim \|f\|_{\tilde{F}^{p_1,q}_{\alpha, \eta}(d\mu_{\nu})} \|g\|_{L^{p_2}(d\mu_{\nu})}, 
\label{eq:claim_2} \\ 
\mathcal{\tilde{F}}^{p,q}_{\alpha, \eta}(\Pi_f(g)) & \lesssim \|f\|_{\tilde{F}^{p_1,q}_{\alpha, \eta}(d\mu_{\nu})} \|g\|_{L^{p_2}(d\mu_{\nu})}, 
 \label{eq:claim_3} \\
\mathcal{\tilde{F}}^{p,q}_{\alpha, \eta}(\Pi(f,g)) & \lesssim \|f\|_{\tilde{F}^{p_1,q}_{\alpha, \eta}(d\mu_{\nu})} \|g\|_{L^{p_2}(d\mu_{\nu})} + \|f\|_{L^{p_3}(d\mu_{\nu})} \|g\|_{\tilde{F}^{p_4,q}_{\alpha, \eta}(d\mu_{\nu})}. 
\label{eq:claim_4}
\end{align}

\medskip 
Clearly, \eqref{eq:claim_1} follows from H\"{o}lder's inequality. Next, we shall prove \eqref{eq:claim_2}.
For the same, making use of Lemma \ref{lem:heat_operator_and_derivative} (with $\kappa' = 2b, \, \kappa = 4b$ and $t_0 = \frac{1}{2}$), we have  
\begin{align*}
|W_t^{(m)} \, W_1^{(h)} (W_1^{(k)} f \, W_1^{(n)} g)| & = |t^m \, G_\nu^{m+h} \, e^{-(t+1)G_\nu} \, (W_1^{(k)} f \, W_1^{(n)} g)| \\
& \lesssim t^m \, (t+1)^{-m-h} \, e^{-b (t+1)G_\nu} \, |W_1^{(k)} f \, W_1^{(n)} g| \\
&\leq t^m \, e^{- 2b^2 \, G_\nu} \, |W_1^{(k)} f \, W_1^{(n)} g| . 
\end{align*}
Using the above estimate along with Lemmas \ref{lem:heat_operator_and_derivative},  \ref{lemma:ball_volume_and_heat_operator_commutation} and the $L^{p}(d\mu_{\nu})$-boundedness of the heat operator, we have 
\begin{align*}
& \mathcal{\tilde{F}}^{p,q}_{\alpha, \eta}(W_1^{(h)} (W_1^{(k)} f \, W_1^{(n)} g))  \\
&= \left\|\left(\int_{0}^{1} (V_{\nu}(t)^{-\alpha/Q}\, V_{\nu}^{\eta} \, |W_t^{(m)} \, W_1^{(h)} (W_1^{(k)} f \, W_1^{(n)} g)|)^{q} \, \frac{dt}{t}\right)^{1/q}\right\|_{L^p(d\mu_{\nu})} \\
& \lesssim \left\|\left(\int_{0}^{1} (V_{\nu}^{-\alpha/Q+\eta} \, t^{m-\alpha/2} \,  e^{-2 b^2 \, G_\nu} \, |W_1^{(k)} f \, W_1^{(n)} g|)^{q} \, \frac{dt}{t} \right)^{1/q}\right\|_{L^p(d\mu_{\nu})} \\
& \lesssim \|V_{\nu}^{-\alpha/Q+\eta} |W_1^{(k)} f \, W_1^{(n)} g|\|_{L^p(d\mu_{\nu})} \\
& \lesssim \|V_{\nu}^{-\alpha/Q+\eta} |W_1^{(k)} f|\|_{L^{p_1}(d\mu_{\nu})} \|W_1^{(n)} g\|_{L^{p_2}(d\mu_{\nu})} \\
& \lesssim \|V_{\nu}^{-\alpha/Q+\eta} f\|_{L^{p_1}(d\mu_{\nu})} \, \|g\|_{L^{p_2}(d\mu_{\nu})} \\
& \lesssim \|f\|_{\tilde{F}^{p_1,q}_{\alpha, \eta}(d\mu_{\nu})} \, \|g\|_{L^{p_2}(d\mu_{\nu})}, 
\end{align*}
which completes the proof of claim \eqref{eq:claim_2}. 

\medskip Now, we move on to proving \eqref{eq:claim_3}. In the following, we have $0 \leq h, k \leq m-1$, and for convenience, we use the notation $F_{t, m, k} = W_t^{(m)} f \, W_t^{(k)} g$. Now, it follows with the help of Lemma \ref{lem:heat_operator_and_derivative} that 
\begin{align*}
|W_u^{(m)} \, W_t^{(h)} \, F_{t, m, k}| = |u^m \, t^{h} G_\nu^{m+h} \, e^{-\left(t + u\right)G_\nu} \, F_{t, m, k}| \lesssim u^{m} \, (t+u)^{-m} \, e^{-b\left(t + u\right)G_\nu} \, |F_{t, m, k}|, 
\end{align*}
which along with Proposition \ref{prop:semigroup_continuous_version_of_Fefferman_Stein_theorem} gives us
\begin{align*}
& \left\|\left(\int_{0}^{1}\left(V_{\nu}(u)^{-\alpha/Q} \, V_{\nu}^{\eta} \, \int_{0}^{1} |W_u^{(m)} \, W_t^{(h)} \, F_{t, m, k}|  \,  \frac{dt}{t} \right)^q \, \frac{du}{u}\right)^{1/q}\right\|_{L^p(d\mu_{\nu})} \\ 
& \lesssim \left\|\left(\int_{0}^{1}\left(V_{\nu}(u)^{-\alpha/Q} \, V_{\nu}^{\eta} \, \int_{0}^{1} u^{m} \, (u+t)^{-m} \, e^{-b \, (t + u) G_\nu} \, |F_{t, m, k}| \,  \frac{dt}{t} \right)^q \, \frac{du}{u}\right)^{1/q}\right\|_{L^p(d\mu_{\nu})} \\ 
& \lesssim \left\|\left(\int_{0}^{1}\left(V_{\nu}(u)^{-\alpha/Q} \, V_{\nu}^{\eta} \, u^m \, e^{-b \, uG_\nu} \, \int_{0}^{1}  (u+t)^{-m} \, e^{-b \, tG_\nu} |F_{t, m, k}| \, \frac{dt}{t}\right)^q \, \frac{du}{u}\right)^{1/q}\right\|_{L^p(d\mu_{\nu})}
\\ & \lesssim  \left\|\left(\int_{0}^{1} \left(V_{\nu}(u)^{-\alpha/Q}  \, V_{\nu}^{\eta} \, u^m \int_{0}^{1}  (u+t)^{-m} \, e^{-b \, tG_\nu} \, |F_{t, m, k}| \, \frac{dt}{t} \right)^q \, \frac{du}{u}\right)^{1/q}\right\|_{L^p(d\mu_{\nu})}. 
\end{align*}
We decompose the inner integral in $t$-variable in two parts, namely, over $(0,u)$ and $(u,1)$. In each of these domains, we make use of the ball volume estimates for $V_{\nu}(t) / V_{\nu}(u)$, so that the above estimate can be further dominated (with the help of Proposition \ref{prop:semigroup_continuous_version_of_Fefferman_Stein_theorem}, Lemma \ref{lem:Cor_of_Schur_lemma}, and boundedness of the heat maximal operator on $L^{p}(d\mu_{\nu})$-spaces) by  
\begin{align*}
&  \left\|\left(\int_{0}^{1}\left(u^{m-\alpha d/2Q}  \int_{0}^{u}  (u+t)^{-m} \,  \left(t^{\alpha d/2Q} \, V_{\nu}(t)^{-\frac{\alpha}{Q}} \, V_{\nu}^{\eta} \, e^{-b \, tG_\nu} \, |F_{t, m, k}|\right) \, \frac{dt}{t} \right)^{q} \frac{du}{u} \right)^{1/q} \right\|_{L^{p}(d\mu_{\nu})} \\
& + \left\|\left(\int_{0}^{1}\left(u^{m-\alpha/2} \int_{u}^{1} (u+t)^{-m} \, \left(t^{\alpha/2} \, V_{\nu}(t)^{-\frac{\alpha}{Q}} \, V_{\nu}^{\eta} \, e^{-b \, tG_\nu} \, |F_{t, m, k}| \right) \, \frac{dt}{t}\right)^q \, \frac{du}{u}\right)^{1/q}\right\|_{L^p(d\mu_{\nu})} \\
& \lesssim \left\|\left(\int_{0}^{1}(V_{\nu}(t)^{-\alpha/Q}\, V_{\nu}^{\eta} \, e^{-b \, tG_\nu} |F_{t, m, k}|)^q \, \frac{dt}{t} \right)^{1/q}\right\|_{L^p(d\mu_{\nu})}
\\ & \lesssim  \left\|\left(\int_{0}^{1}(V_{\nu}(t)^{-\alpha/Q}\, V_{\nu}^{\eta} \, |W_t^{(m)} f \, W_t^{(k)} g|)^q \, \frac{dt}{t} \right)^{1/q} \right\|_{L^p(d\mu_{\nu})} \\  
& \lesssim \|\sup_{t \in (0,1)} e^{-b \, tG_\nu} |g|\|_{L^{p_2}(d\mu_{\nu})} \left\|\left(\int_{0}^{1}(V_{\nu}(t)^{-\alpha/Q}\, V_{\nu}^{\eta} \, | W_t^{(m)} f|)^{q} \frac{dt}{t} \right)^{1/q}\right\|_{L^{p_1}(d\mu_{\nu})}
\\  & \lesssim \|f\|_{\tilde{F}^{p_1,q}_{\alpha, \eta}(d\mu_{\nu})} \|g\|_{L^{p_2}(d\mu_{\nu})}, 
\end{align*}
and this completes the proof of claim \eqref{eq:claim_3}. 

\medskip 
Finally, to establish claim \eqref{eq:claim_4}, note that $\mathcal{\tilde{F}}^{p,q}_{\alpha, \eta}(\Pi(f,g))$ is dominated by a finite sum of terms of the following type:
\begin{align} \label{term_in_last_claim}
\left\|\left(\int_{0}^{1} \left(V_{\nu}(u)^{-\alpha/Q} \, V_{\nu}^{\eta} \, \int_{0}^{1} |W_{u}^{(m)} W_{t}^{(m)} (W_{t}^{(h)}f \, W_{t}^{(k)}g)| \, \frac{dt}{t} \right)^{q} \frac{du}{u}\right)^{1/q}\right\|_{L^{p}(d\mu_{\nu})}.    
\end{align} 
Now, by using Lemma \ref{lem:heat_operator_and_derivative}, we have 
\begin{align*}
&|W_u^{(m)} \, W_t^{(m)} \, (W_t^{(h)} f \, W_t^{(k)} g)| \\
& = |u^m \, G_\nu^m \, e^{-(t+u)G_\nu} \, (tG_\nu)^m \, (W_t^{(h)} f \, W_t^{(k)} g)| \\
& \lesssim u^m \, (u+t)^{-m} \, e^{-b \, (u+t) G_\nu} \, |(tG_\nu)^m \, (W_t^{(h)} f \, W_t^{(k)} g)| \\
& \lesssim  u^m \, (u+t)^{-m} \, e^{-b \, (u+t) G_\nu} \, t^{m+h+k} \sum_{|\gamma| + |\beta|= 2(m+h+k)} |(X^{\gamma} e^{-tG_\nu} f) \, (X^{\beta} e^{-tG_\nu} g)|, 
\end{align*}
so it suffices to work with one such term in $(\gamma, \, \beta)$. More precisely, if we write $F_{t, \gamma, \beta} = t^{m+h+k} (X^{\gamma} e^{-tG_\nu} f) \, (X^{\beta} e^{-tG_\nu} g)$, then repeating the arguments performed in the proof of claim \ref{eq:claim_3}, we would get that a term of the type \eqref{term_in_last_claim} is bounded by a finite sum of terms of the following form:
\begin{align} 
\label{eq:claim_4-est-1} 
& \left\|\left(\int_{0}^{1} \left( V_{\nu}(u)^{-\alpha/Q} \, V_{\nu}^{\eta} \, e^{- b u G_\nu} \, \int_{0}^{1} u^m \,  (u+t)^{-m} e^{-b \, t G_\nu} |F_{t, \gamma, \beta}| \, \frac{dt}{t}\right)^{q} \, \frac{du}{u}\right)^{1/q} \right\|_{L^p(d\mu_{\nu})} \\ 
\nonumber & \lesssim \left\|\left(\int_{0}^{1} \left( V_{\nu}(t)^{-\alpha/Q} \, V_{\nu}^{\eta} \, |F_{t, \gamma, \beta}| \right)^{q} \frac{dt}{t} \right)^{1/q} \right\|_{L^p(d\mu_{\nu})}.
\end{align} 

We shall estimate \eqref{eq:claim_4-est-1} in three different cases: when $\gamma =0$, when $\beta =0$, and when both $\beta$ and $\gamma$ are non-zero. Let us first consider the case when $\gamma =0$. In this case, \eqref{eq:claim_4-est-1} can be dominated by 
\begin{align*}
& \left\|(\sup_{t \in (0,1)} |e^{-tG_{\nu}}f|)\left(\int_{0}^{1} (V_{\nu}(t)^{-\alpha/Q} \, V_{\nu}^{\eta} \, |t^{m+h+k} \, X^{\beta} e^{-tG_{\nu}}g| )^{q} \, \frac{dt}{t} \right)^{1/q}\right\|_{L^p(d\mu_{\nu})} \\ 
& \lesssim \left\|\sup_{t \in (0,1)} |e^{-tG_{\nu}}f|\right\|_{L^{p_3}(d\mu_{\nu})} \left\|\left(\sum_{j=1}^{\infty} \left( V_{\nu}(2^{-j})^{-\alpha/Q} \, V_{\nu}^{\eta} \, W_{2^{-j}}^{(|\beta|/2), \ast} g \right)^{q} \right)^{1/q}\right\|_{L^{p_4}(d\mu_{\nu})} \\ 
& \lesssim \|f\|_{L^{p_3}(d\mu_{\nu})} \, \|g\|_{\tilde{F}^{p_4,q}_{\alpha, \eta}(d\mu_{\nu})}, 
\end{align*}
where the last step uses Theorem \ref{thm:vector_field_characterisation_Triebel_Lizorkin_norm}, and with this the claim \eqref{eq:claim_4} is proved in the case when $\gamma = 0$. The case when $\beta=0$ can be handled in the same manner. 

\medskip 
Finally, let us work when $\beta \neq 0 \neq \gamma$. In this case, let us write $\theta = \frac{|\gamma|}{|\gamma|+|\beta|}$. Clearly, $\theta \in (0,1)$. Let us also write $\alpha_1 = \theta \alpha$ and $\alpha_2 = (1-\theta) \alpha$ so that $\alpha = \alpha_1 + \alpha_2$. Then, final estimate in \eqref{eq:claim_4-est-1} can be dominated by 
\begin{align}
\label{eq:claim_4-est-2} 
\nonumber & \left\|\left(\int_{0}^{1} \left(V_{\nu}(t)^{-\alpha/Q} \, V_{\nu}^{\eta} \, |F_{t, \gamma, \beta}| \right)^{q} \frac{dt}{t} \right)^{1/q} \right\|_{L^p(d\mu_{\nu})} \\
\nonumber & \lesssim \left\|\left(\sum_{j=1}^{\infty} (V_{\nu}(2^{-j})^{-\alpha/Q} \, V_{\nu}^{\eta} \, W_{2^{-j}}^{(|\gamma|/2), \ast} f \, W_{2^{-j}}^{(|\beta|/2), \ast} g )^{q} \, \right)^{1/q} \right\|_{L^p(d\mu_{\nu})} \\
\nonumber & \lesssim \left\|\left(\sum_{j=1}^{\infty} (V_{\nu}(2^{-j})^{-\alpha_1/Q} \, V_{\nu}^{\eta\theta} \, W_{2^{-j}}^{(|\gamma|/2), \ast} f)^{q_1} \right)^{1/q_1}\right\|_{L^{r_1}(d\mu_{\nu})} \\
\nonumber & \quad \times \left\|\left(\sum_{j=1}^{\infty} (V_{\nu}(2^{-j})^{-\alpha_2/Q}  \, V_{\nu}^{\eta (1-\theta)} \, W_{2^{-j}}^{(|\beta|/2), \ast} g)^{q_2} \right)^{1/q_2}\right\|_{L^{r_2}(d\mu_{\nu})} \\
& \lesssim \|f\|_{\tilde{F}^{r_1, q_1}_{\alpha_1, \eta \theta}(d\mu_{\nu})} \, \|g\|_{\tilde{F}^{r_2, q_2}_{\alpha_2, \eta (1-\theta)}(d\mu_{\nu})},
\end{align}
where the H\"{o}lder's inequality is applied for the pairs $(q_1, q_2)$ and $(r_1, r_2)$ where $\frac{1}{q_1} =\frac{\theta}{q}, \, \frac{1}{q_2} =\frac{1-\theta}{q}, \,  \frac{1}{r_1}= \frac{\theta}{p_1}+\frac{1-\theta}{p_3},$ and $ \frac{1}{r_2} = \frac{\theta}{p_2} +\frac{1-\theta}{p_4}.$

\medskip 
We can now invoke the complex interpolation Theorem \ref{thm:interpolation_Triebel_Lizorkin_spaces} to have 
\begin{align*}
\left( \tilde{F}^{p_3, \infty}_{0,0}(d\mu_{\nu}), \tilde{F}^{p_1, q}_{\alpha,\eta}(d\mu_{\nu}) \right)_{[\theta]} & = \tilde{F}^{r_1, q_1}_{\alpha_1, \eta \theta}(d\mu_{\nu}), \\ 
\quad \text{and} \qquad 
\left( \tilde{F}^{p_4, q}_{\alpha, \eta}(d\mu_{\nu}), \tilde{F}^{p_2, \infty}_{0,0}(d\mu_{\nu}) \right)_{[\theta]} & = \tilde{F}^{r_2, q_2}_{\alpha_2, \eta (1-\theta)}(d\mu_{\nu}),
\end{align*}
which implies that 
\begin{align*}
\|f\|_{\tilde{F}^{r_1, q_1}_{\alpha_1, \eta \theta}(d\mu_{\nu})} \, \|g\|_{\tilde{F}^{r_2, q_2}_{\alpha_2, \eta (1-\theta)}(d\mu_{\nu})} & \lesssim \|f\|_{\tilde{F}^{p_3, \infty}_{0, 0}(d\mu_{\nu})}^{1-\theta} \, \|f\|_{\tilde{F}^{p_1, q}_{\alpha, \eta}(d\mu_{\nu})}^{\theta} \, \|g\|_{\tilde{F}^{p_4, q}_{\alpha, \eta}(d\mu_{\nu})}^{1-\theta}  \, \|g\|_{\tilde{F}^{p_2, \infty}_{0, 0}(d\mu_{\nu})}^{\theta} \\
& \lesssim \|f\|_{L^{p_3}(d\mu_{\nu})}^{1-\theta} \, \|f\|_{\tilde{F}^{p_1, q}_{\alpha, \eta}(d\mu_{\nu})}^{\theta} \, \|g\|_{\tilde{F}^{p_4, q}_{\alpha, \eta}(d\mu_{\nu})}^{1-\theta} \, \|g\|_{L^{p_2}(d\mu_{\nu})}^{\theta} \\
& \lesssim \|f\|_{L^{p_3}(d\mu_{\nu})} \, \|g\|_{\tilde{F}^{p_4, q}_{\alpha, \eta}(d\mu_{\nu})} +  \|f\|_{\tilde{F}^{p_1, q}_{\alpha, \eta}(d\mu_{\nu})} \, \|g\|_{L^{p_2}(d\mu_{\nu})},
\end{align*}
and putting this in \eqref{eq:claim_4-est-2}, the claim \eqref{eq:claim_4} is established in the case when $\beta \neq 0 \neq \gamma$. 

\medskip 
This completes the proof of all four claims \eqref{eq:claim_1}--\eqref{eq:claim_4} and hence the theorem.
\end{proof}


\section*{Acknowledgments}
Second author is grateful to Indian Institute of Science Education and Research (IISER) Bhopal for the Senior Research Fellowship (SRF). First and third authors were partially supported by the Anusandhan National Research Foundation (ANRF), India, under the research project ANRF/ARG/2025/003732/MS.  


\providecommand{\bysame}{\leavevmode\hbox to3em{\hrulefill}\thinspace}
\providecommand{\MR}{\relax\ifhmode\unskip\space\fi MR }
\providecommand{\MRhref}[2]{%
  \href{http://www.ams.org/mathscinet-getitem?mr=#1}{#2}
}
\providecommand{\href}[2]{#2}


\end{document}